\documentclass[]{siamltex}

\usepackage{epsfig}

\usepackage{amsfonts}
\usepackage{amsmath}
\usepackage{algorithm}
\usepackage{algorithmic}
\usepackage{multirow}

\graphicspath{{figures/}}

\newcommand{\Xe}{{X_{\rm e}}}
\newcommand{\Ye}{{Y_{\rm e}}}
\newcommand{\Ze}{{Z_{\rm e}}}
\newcommand{\Xpe}{{X'_{\rm e}}}
\newcommand{\Ype}{{Y'_{\rm e}}}

\newcommand{\LB}{{\rm LB}}
\newcommand{\UB}{{\rm UB}}
\newcommand{\cN}{{\mathcal N}}
\newcommand{\cD}{{\mathcal D}}
\newcommand{\cT}{{\mathcal T}}
\newcommand{\bNmax}{{\mathbb{N}_{\rm max}}}
\newcommand{\bK}{{\mathbb{K}}}

\title{Reduced Basis {\em A Posteriori} Error Bounds for the Instationary Stokes Equations}
\author{Anna-Lena Gerner\thanks{Aachen Institute for Advanced Study in Computational Engineering Science (AICES), RWTH Aachen University, Schinkelstra\ss e 2, 52062 Aachen, Germany ({\tt gerner@aices.rwth-aachen.de})}
\and Karen Veroy\thanks{Aachen Institute for Advanced Study in Computational Engineering Science (AICES) and Faculty of Civil Engineering, RWTH Aachen University, Schinkelstra\ss e 2, 52062 Aachen, Germany ({\tt veroy@aices.rwth-aachen.de})}}

\begin{document}

\maketitle

\begin{abstract}
We present reduced basis approximations and rigorous {\em a posteriori} error bounds for the instationary Stokes equations.  
We shall discuss both a method based on the standard formulation as well as a method based on a penalty approach, which combine techniques developed in \cite{Gerner:2011fk,Gerner:2012fk} and \cite{veroy10:_stokes_penalty} with current reduced basis techniques for parabolic problems. The analysis then shows how time integration affects the development of reduced basis {\em a posteriori} error bounds as well as the construction of computationally efficient reduced basis approximation spaces. 
To demonstrate their performance in practice, the methods are applied to a Stokes flow in a two-dimensional microchannel with a parametrized rectangular obstacle; evolution in time is induced by a time-dependent velocity profile on the inflow boundary.
Numerical results illustrate (i)~the rapid convergence of reduced basis approximations, (ii)~the performance of {\em a posteriori} error bounds with respect to sharpness, and (iii)~computational efficiency. 
\end{abstract}



\pagestyle{myheadings}
\thispagestyle{plain}
\markboth{\sc A.-L.~Gerner and K.~Veroy}{\sc Reduced Basis Methods for the Instationary Stokes Equations}

\section*{Introduction}

Designed for the real-time and many-query context of parameter estimation, optimization, and control, the reduced basis (RB) method permits the efficient yet reliable approximation of input-output relationships induced by parametrized partial differential equations. The essential ingredients are: (i) dimension reduction, through Galerkin projection onto a low-dimensional RB space; (ii) certainty, through rigorous {\em a posteriori} bounds for the errors in the RB approximations; (iii) computational efficiency, through an Offline-Online computational strategy; and (iv) effectiveness, through a greedy sampling approach.

In this paper, we demonstrate how RB techniques presented in \cite{Gerner:2011fk,Gerner:2012fk,veroy10:_stokes_penalty} for parametrized saddle point problems may be extended to the time-dependent setting. To this end, we consider the instationary Stokes equations. We shall discuss both a method based on the standard formulation as well as a method based on a penalty approach (see also \cite{Gerner:2012uq} for initial results), which combine techniques developed in \cite{Gerner:2011fk,Gerner:2012fk} and \cite{veroy10:_stokes_penalty} with current RB techniques for parabolic problems (see, e.g., \cite{Grepl:2005kx,grepl05,haasdonk08:_reduc}). The analysis then shows how time integration affects the development of RB {\em a posteriori} error bounds as well as the construction of computationally efficient RB approximation spaces. 

Starting from the standard mixed formulation of the instationary Stokes equations, we develop rigorous {\em a posteriori} error bounds for the RB velocity approximations. As in the stationary case presented in \cite{Gerner:2011fk,Gerner:2012fk}, they involve the (Online-) estimation of coercivity, continuity, and inf-sup stability constants associated with the diffusion term and incompressibility constraint; in addition, they now also depend on continuity constants associated with the mass term.
Employing a penalty formulation, we obtain rigorous upper bounds for the errors in both the velocity and pressure approximations. As in the stationary case presented in \cite{veroy10:_stokes_penalty}, they are computationally very efficient since they do not involve the estimation of inf-sup stability constants but depend only on coercivity constants associated with the diffusion and penalty terms; however, they again also depend on the penalty parameter such that associated effectivities increase as we approach the nonpenalized problem.
To construct efficient RB approximation spaces, we consider a POD greedy procedure (see \cite{Grepl:2005kx,Haasdonk:2011ffk,haasdonk08:_reduc}) that is coupled with adaptive stabilization techniques developed in \cite{Gerner:2011fk}.
To demonstrate their performance in practice, the methods are then applied to a Stokes flow in a parametrized domain where evolution in time is induced by a time-dependent velocity profile on the inflow boundary.

This paper is organized as follows: In \S \ref{s:general_problem_statement}, we introduce the general problem formulation and its ``truth'' approximation upon which our RB approximation will subsequently be built. We start from a time-discrete framework already that allows us to directly recover the settings discussed in \cite{Gerner:2011fk,Gerner:2012fk} and \cite{veroy10:_stokes_penalty}; now, we have a saddle point problem associated with each time step. The time discretization scheme is given by a backward Euler method. Section \ref{s:RB_method} then describes our RB method. In \S \ref{ss:Galerkin_projection}, we define the RB approximation as the Galerkin projection onto a low-dimensional RB approximation space. We develop rigorous {\em a posteriori} error bounds in \S \ref{ss:error_estimation}. Both RB approximations and error bounds can be computed Online-efficiently as summarized in \S \ref{ss:offline_online}. This enables us to employ adaptive sampling processes for constructing computationally efficient RB approximation spaces, which shall be outlined in \S \ref{ss:construction_RB_spaces}. In \S \ref{s:model_problem}, we introduce our instationary Stokes model problem. Numerical results in \S \ref{s:numerical_results} then illustrate (i)~the rapid convergence of RB approximations, (ii)~the performance of {\em a posteriori} error bounds with respect to sharpness, and (iii)~computational efficiency. Finally, in \S \ref{s:conclusion}, we give some concluding remarks.

\section{General Problem Statement} \label{s:general_problem_statement}

\subsection{Formulation} \label{ss:pspp}

Let $\Xe$ and $\Ye$ be two Hilbert spaces with inner products $(\cdot,\cdot)_{\Xe}$, $(\cdot,\cdot)_{\Ye}$ and associated norms $\|\cdot\|_{\Xe} = \sqrt{(\cdot,\cdot)_{\Xe}}$, $\|\cdot\|_{\Ye}=\sqrt{(\cdot,\cdot)_{\Ye}}$, respectively.\footnote{Here and in the following, the subscript e denotes ``exact''.} We define the product space $\Ze \equiv \Xe \times \Ye$, with inner product $(\cdot,\cdot)_{\Ze} \equiv (\cdot,\cdot)_{\Xe} + (\cdot,\cdot)_{\Ye}$ and norm $\|\cdot\|_{\Ze} = \sqrt{(\cdot,\cdot)_{\Ze}}$. The associated dual spaces are denoted by $\Xpe$, $\Ype$, and $Z'_{\rm e}$.

Furthermore, let $\cD \subset \mathbb{R}^{n}$ be a prescribed $n$-dimensional, compact parameter set.  For any parameter $\mu \in \cD$, we then consider the continuous bilinear forms $m(\cdot,\cdot;\mu):\Xe\times \Xe \to\mathbb{R}$, 
$a(\cdot,\cdot;\mu):\Xe\times \Xe\rightarrow \mathbb{R}$, and $b(\cdot,\cdot;\mu):\Xe\times \Ye \rightarrow \mathbb{R}$,\footnote{For clarity of exposition, we suppress the obvious requirement of nonzero elements in the denominators.} 
\begin{align}  \label{eq:gamma_m_e}
\gamma_m^{\rm e}(\mu) & \equiv \sup_{u\in \Xe} \sup_{v\in \Xe}\frac{m(u,v;\mu)}{\|u\|_{\Xe}\|v\|_{\Xe}} < \infty \quad\forall\;\mu\in\cD,\\ 
\label{eq:gamma_a_e}
\gamma_a^{\rm e}(\mu) & \equiv \sup_{u\in \Xe} \sup_{v\in \Xe}\frac{a(u,v;\mu)}{\|u\|_{\Xe}\|v\|_{\Xe}} < \infty \quad\forall\;\mu\in\cD,\\
\label{eq:gamma_b_e}
\gamma_b^{\rm e}(\mu) &\equiv \sup_{q\in \Ye} \sup_{v\in \Xe} \frac{b(v,q;\mu)}{\|q\|_{\Ye}\|v\|_{\Xe}} < \infty \quad\forall\;\mu\in\cD,
\end{align}
as well as $c(\cdot,\cdot;\mu):\Ye\times\Ye\to\mathbb{R}$,
\begin{align}
\label{eq:gamma_c_e}
\gamma_c^{\rm e}(\mu) & \equiv \sup_{p\in \Ye} \sup_{q\in \Ye}\frac{c(p,q;\mu)}{\|p\|_{\Ye}\|q\|_{\Ye}} < \infty \quad\forall\;\mu\in\cD.
\end{align}
We moreover assume that $a(\cdot,\cdot;\mu)$ and $c(\cdot,\cdot;\mu)$ are coercive on $\Xe$ and $\Ye$, respectively,
\begin{align} \label{eq:alpha_a_e}
\alpha_a^{\rm e}(\mu) &\equiv \inf_{v\in \Xe} \frac{a(v,v;\mu)}{\|v\|^2_{\Xe}} > 0  \quad\forall\;\mu\in\cD,\\
 \label{eq:alpha_c_e}
\alpha_c^{\rm e}(\mu) &\equiv \inf_{q\in \Ye} \frac{c(q,q;\mu)}{\|q\|^2_{\Ye}} > 0 \quad\forall\;\mu\in\cD,
\end{align}
$m(\cdot,\cdot;\mu)$ is symmetric and positive definite,
\begin{align}\label{eq:pos_def_m_e}
m(v,v;\mu) &> 0 \quad\forall\; 0\neq v\in\Xe \quad\forall\;\mu\in\cD,
\end{align}
and $b(\cdot,\cdot;\mu)$ satisfies the inf-sup condition
\begin{equation} \label{eq:brezzi_infsup_e}
\beta^{\rm e}(\mu) \equiv \inf_{q\in \Ye} \sup_{v\in \Xe} \frac{b(v,q;\mu)}{\|q\|_{\Ye}\|v\|_{\Xe}}>0\quad\forall\; \mu\in\cD.
\end{equation}
By (\ref{eq:gamma_a_e}), (\ref{eq:alpha_a_e}) and (\ref{eq:gamma_c_e}), (\ref{eq:alpha_c_e}), the bilinear forms $a(\cdot,\cdot;\mu)$ and $c(\cdot,\cdot;\mu)$ provide with $\|\cdot\|_{\Xe,\mu}\equiv\sqrt{a(\cdot,\cdot;\mu)}$ and $\|\cdot\|_{\Ye,\mu}\equiv\sqrt{c(\cdot,\cdot;\mu)}$ energy norms on $\Xe$ and $\Ye$, respectively, which are equivalent to $\|\cdot\|_\Xe$ and $\|\cdot\|_\Ye$ for any $\mu\in\cD$; note that, to this end, $a(\cdot,\cdot;\mu)$ and $c(\cdot,\cdot;\mu)$ do not necessarily have to be symmetric.
Furthermore, as a symmetric and positive definite bilinear form, $m(\cdot,\cdot;\mu)$ defines an inner product on $\Xe$ for any parameter $\mu\in\cD$; the associated norm shall be denoted by $\|\cdot\|_\mu \equiv \sqrt{m(\cdot,\cdot;\mu)}$. 

We further assume that we are given a time interval $[0,T]$, $T>0$, and linear functionals $f(\cdot;\mu)\in C^0(0,T;\Xpe)$ and $g(\cdot;\mu)\in C^0(0,T;\Ype)$ for all $\mu\in\cD$; for a vector space $V$, $C^0(0,T;V)$ here denotes the space of $V$-valued functions of class $C^0$ with respect to $t\in [0,T]$.
Throughout this work, we directly consider a time-discrete framework: We divide the time interval $[0,T]$ into $K$ subintervals of equal length $\Delta t \equiv T/K$, and define $t^k\equiv k\Delta t$ for all $k=0,\ldots, K$; for notational convenience, we also introduce $\bK \equiv \{1,\ldots, K\}$ and $\bK_0 \equiv \bK \cup \{0\}$. We then set
$f^k(\cdot;\mu) \equiv f(t^k;\mu)\in \Xpe$ and $g^k(\cdot;\mu) \equiv g(t^k;\mu)\in \Ype$ for all $k\in\bK_0$, $\mu\in\cD$.

For $\varepsilon\geq 0$, we now consider the following ``exact''---more precisely, semi-discrete---problem resulting from a backward Euler method (see, e.g., \cite{Ern:2004jl,Reusken:2011fk,Quarteroni:2008fk,Thomee:1997fk}): For any given parameter $\mu\in\cD$, we find $u^{\varepsilon,k}_{\rm e}(\mu)\in \Xe$ and $p^{\varepsilon,k}_{\rm e}(\mu)\in \Ye$, $k\in\bK$, such that $u^{\varepsilon,0}_{\rm e}(\mu)=0$\footnote{We here assume zero initial conditions for simplicity; note that nonzero initial conditions can be handled as well without much difficulty (see \cite{grepl05}).} and
\begin{align}\nonumber
&{\textstyle \frac{1}{\Delta t}}\, m(u^{\varepsilon,k}_{\rm e}(\mu)-u^{\varepsilon,k-1}_{\rm e}(\mu),v;\mu) \\ \label{eq:exact_scheme}
&\qquad\begin{split}
 + \; a(u^{\varepsilon,k}_{\rm e}(\mu),v;\mu) + b(v,p^{\varepsilon,k}_{\rm e}(\mu);\mu) &= f^k(v;\mu) \quad\forall\;v\in \Xe,\\
 b(u^{\varepsilon,k}_{\rm e}(\mu),q;\mu) -\varepsilon\,c(p^{\varepsilon,k}_{\rm e}(\mu),q;\mu)&= g^k(q;\mu) \hspace{0.3ex}\quad\forall\;q\in \Ye,
 \end{split}
 \quad k\in\bK.
\end{align}
Even though we here use a common notation for simplicity in exposition, we point out that (\ref{eq:exact_scheme}) states very different problems for $\varepsilon=0$ and $\varepsilon>0$, respectively. For $\varepsilon=0$, we also denote $u^k_{\rm e}(\mu)\equiv u^{0,k}_{\rm e}(\mu)$, $k\in\bK_0$, and $p^k_{\rm e}(\mu)\equiv p^{0,k}_{\rm e}(\mu)$, $k\in\bK$, for all $\mu\in\cD$. For $\varepsilon>0$, corresponding to our discussions in \cite{veroy10:_stokes_penalty}, (\ref{eq:exact_scheme}) can be considered as a perturbed or regularized version of the problem associated with $\varepsilon=0$; in this case, we therefore call $(u^{\varepsilon,k}_{\rm e}(\mu), p^{\varepsilon,k}_{\rm e}(\mu)), k\in\bK$, also the penalty solution.
Since these problems differ considerably in their general nature (cf.~\cite{Gerner:2011fk} and \cite{veroy10:_stokes_penalty}), we shall often treat them separately in the following analysis and explicitly distinguish between the two cases $\varepsilon=0$ and $\varepsilon>0$.
From (\ref{eq:alpha_a_e}) and (\ref{eq:pos_def_m_e}), the bilinear form $\frac{1}{\Delta t}m(\cdot,\cdot;\mu) + a(\cdot,\cdot;\mu)$ is coercive on $\Xe$ for any $\mu\in\cD$. The problem (\ref{eq:exact_scheme}) is thus uniquely solvable for $(u^{k}_{\rm e}(\mu), p^{k}_{\rm e}(\mu)), k\in\bK$, and $(u^{\varepsilon,k}_{\rm e}(\mu), p^{\varepsilon,k}_{\rm e}(\mu)), k\in\bK$, as a saddle point problem according to \cite{Gerner:2011fk} and \cite{veroy10:_stokes_penalty}, respectively.

\subsection{Truth Approximation} \label{ss:truth}
We now introduce a high-fidelity ``truth'' approximation upon which our RB approximation will subsequently be built.  To this end, let $X$ and $Y$ denote finite-dimensional subspaces of $\Xe$ and $\Ye$, respectively.  We define the product space $Z \equiv X \times Y$ and denote by $\cN$ the dimension of $Z$. We emphasize that the dimension $\cN$ is typically very large. These ``truth'' approximation subspaces inherit the inner products and norms of the exact spaces: $(\cdot,\cdot)_X \equiv (\cdot,\cdot)_{\Xe}$, $\|\cdot\|_X\equiv \|\cdot\|_{\Xe}$, $(\cdot,\cdot)_Y \equiv (\cdot,\cdot)_{\Ye}$, $\|\cdot\|_Y\equiv \|\cdot\|_{\Ye}$, and $(\cdot,\cdot)_Z \equiv (\cdot,\cdot)_{\Ze}$, $\|\cdot\|_Z\equiv \|\cdot\|_{\Ze}$.

Clearly, the continuity properties (\ref{eq:gamma_m_e}), (\ref{eq:gamma_a_e}), (\ref{eq:gamma_b_e}), and (\ref{eq:gamma_c_e}) are passed on to the ``truth'' approximation spaces,
\begin{align} \label{eq:gamma_m}
\gamma_m(\mu) &\equiv \sup_{u\in X} \sup_{v\in X} \frac{m(u,v;\mu)}{\|u\|_X \|v\|_X} < \infty \quad\forall\;\mu\in\cD,\\ 
\label{eq:gamma_a}
\gamma_a(\mu) &\equiv \sup_{u\in X} \sup_{v\in X} \frac{a(u,v;\mu)}{\|u\|_X \|v\|_X} < \infty \quad\forall\;\mu\in\cD,\\ \label{eq:gamma_b}
\gamma_b(\mu) &\equiv \sup_{q\in Y} \sup_{v\in X} \frac{b(v,q;\mu)}{\|q\|_Y \|v\|_X} < \infty\quad\forall\;\mu\in\cD,\\
\label{eq:gamma_c}
\gamma_c(\mu) &\equiv \sup_{p\in X} \sup_{q\in X} \frac{c(p,q;\mu)}{\|p\|_Y \|q\|_Y} < \infty\quad\forall\;\mu\in\cD;
\end{align}
so are the coercivity properties (\ref{eq:alpha_a_e}) and (\ref{eq:alpha_c_e}),
\begin{align}
 \label{eq:alpha_a}
\alpha_a(\mu) &\equiv \inf_{v\in X} \frac{a(v,v;\mu)}{\|v\|^2_X} > 0 \quad\forall\;\mu\in\cD,\\
 \label{eq:alpha_c}
\alpha_c(\mu) &\equiv \inf_{q\in X} \frac{c(q,q;\mu)}{\|q\|^2_Y} > 0 \quad\forall\;\mu\in\cD,
\end{align}
as well as the inner product $m(\cdot,\cdot;\mu)$ and associated norm $\|\cdot\|_\mu$,
\begin{equation} \label{eq:pos_def_m}
m(v,v;\mu) > 0 \quad\forall\; 0\neq v\in X \quad\forall\;\mu\in\cD.
\end{equation}
Thus, $\|\cdot\|_{X,\mu}\equiv\|\cdot\|_{\Xe,\mu}$ and $\|\cdot\|_{Y,\mu}\equiv\|\cdot\|_{\Ye,\mu}$ define norms on $X$ and $Y$, respectively, which are equivalent to $\|\cdot\|_X$ and $\|\cdot\|_Y$ for any $\mu\in\cD$. We further assume that the approximation spaces $X$ and $Y$ are chosen such that they satisfy the Ladyzhenskaya--Babu\v{s}ka--Brezzi (LBB) inf-sup condition (see, e.g., \cite{Brezzi:1991fk})
\begin{equation} \label{eq:LBB}
\beta(\mu) \equiv \inf_{q\in Y} \sup_{v\in X} \frac{b(v,q;\mu)}{\|q\|_{Y} \|v\|_X} \geq \beta^0(\mu) > 0 \quad \forall\; \mu\in\cD,
\end{equation}
where $\beta^0(\mu)$ is a constant independent of the dimension ${\mathcal{N}}$.

Our high-fidelity ``truth'' discretization for (\ref{eq:exact_scheme}) now reads as follows: For $\varepsilon\geq 0$ and any given $\mu\in\cD$, we find $u^{\varepsilon,k}(\mu)\in X$ and $p^{\varepsilon,k}(\mu)\in Y$, $k\in\bK$, such that $u^{\varepsilon,0}(\mu)=0$ and
\begin{align}\nonumber
&{\textstyle \frac{1}{\Delta t}}\, m(u^{\varepsilon,k}(\mu)-u^{\varepsilon,k-1}(\mu),v;\mu) \\ \label{eq:truth_scheme}
&\qquad\begin{split}
 + \; a(u^{\varepsilon,k}(\mu),v;\mu) + b(v,p^{\varepsilon,k}(\mu);\mu) &= f^k(v;\mu) \quad\forall\;v\in X,\\ 
 \hspace{2ex} b(u^{\varepsilon,k}(\mu),q;\mu) -\varepsilon\,c(p^{\varepsilon,k}(\mu),q;\mu) &= g^k(q;\mu) \hspace{0.2ex}\quad\forall\;q\in Y,
\end{split}
\quad k\in\bK.
\end{align}
In case of $\varepsilon=0$, we also denote $u^k(\mu)\equiv u^{0,k}(\mu)$, $k\in\bK_0$, and $p^k(\mu)\equiv p^{0,k}(\mu)$, $k\in\bK$.
As the exact problem in \S \ref{ss:pspp}, the problem (\ref{eq:truth_scheme}) is uniquely solvable for $(u^{k}(\mu),p^{k}(\mu)), k\in\bK$, and $(u^{\varepsilon,k}(\mu), p^{\varepsilon,k}(\mu)), k\in\bK$, according to \cite{Gerner:2011fk} and \cite{veroy10:_stokes_penalty}, respectively.

\begin{remark} \label{rmrk:truth_LBB} 
We note that in case of $\varepsilon>0$, the LBB inf-sup condition (\ref{eq:LBB}) is in fact not a compulsory requirement for the system (\ref{eq:truth_scheme}) to be well-posed (see, e.g., \cite{Brezzi:1991fk}). However, if the problem is considered as a perturbation of the problem associated with $\varepsilon=0$, the condition is needed for the solution $(u^{\varepsilon,k}(\mu), p^{\varepsilon,k}(\mu)), k\in\bK$, to converge to $(u^{k}(\mu),p^{k}(\mu)), k\in\bK$, as $\varepsilon$ tends to zero (see, e.g., \cite{Bercovier:1978kx}). For further details in this context, we refer the reader also to \cite[\S 4]{Gerner:2012ag}. 
\end{remark}

\section{The Reduced Basis Method} \label{s:RB_method}

We now turn to the RB method, discussing the approximation procedure, rigorous {\em a posteriori} error estimators, and the construction of stable approximation spaces that capture the causality associated with the parameter dependence {\em as well as} with evolution in time.

\subsection{Galerkin Projection}\label{ss:Galerkin_projection}

Suppose that we are given a set of nested, low-dimensional RB approximation subspaces $X_N\subset X_{N+1}\subset X$ and $Y_N\subset Y_{N+1}\subset Y$, $N\in \bNmax \equiv\{1,\ldots,N_{\rm max}\}$. We denote by $N_X$ and $N_Y$ the dimensions of $X_N$ and $Y_N$, respectively, and the total dimension of $Z_N\equiv X_N\times Y_N$ by $N_Z \equiv N_X + N_Y$. The subspaces $X_N$, $Y_N$, and $Z_N$ again inherit all inner products and norms of $X$, $Y$, and $Z$, respectively. The RB approximation is then defined as the Galerkin projection with respect to the truth problem (\ref{eq:truth_scheme}) onto these low-dimensional subspaces: 
For $\varepsilon\geq 0$ and any given $\mu\in \cD$, we find $u^{\varepsilon,k}_N(\mu)\in X_N$ and $p^{\varepsilon,k}_N(\mu)\in Y_N$, $k\in\bK$, such that $u^{\varepsilon,0}_N(\mu)=0$ and
\begin{align}\nonumber
&{\textstyle\frac{1}{\Delta t}}\, m(u^{\varepsilon,k}_N(\mu)-u^{\varepsilon,k-1}_N(\mu),v_N;\mu)\\ \label{eq:rb_scheme}
&\begin{split}
 + \; a(u_{N}^{\varepsilon,k}(\mu),v_N;\mu) + b(v_N,p^{\varepsilon,k}_{N}(\mu);\mu) &= f^k(v_N;\mu)\quad\forall\;v_N\in X_N,\\
b(u^{\varepsilon,k}_{N}(\mu),q_N;\mu) - \varepsilon\, c(p^{\varepsilon,k}_{N}(\mu),q_N;\mu) & = g^k(q_N;\mu)\hspace{0.3ex}\quad\forall\;q_N\in Y_N,
\end{split}
\quad k\in\bK.
\end{align}
Again, we denote $u^k_N(\mu)\equiv u^{0,k}_N(\mu)$, $k\in\bK_0$, and $p^k_N(\mu)\equiv p^{0,k}_N(\mu)$, $k\in\bK$.

The discrete RB system now essentially behaves as in the stationary case: We recall (see \cite{Gerner:2011fk}) that a pair $(X_N,Y_N)$ of RB approximation spaces is called {\em stable} if it satisfies the inf-sup condition
\begin{equation}\label{eq:rb_infsup}
\beta_N(\mu)\equiv \inf_{q_N\in Y_N} \sup_{v_N \in X_N} \frac{b(v_N,q_N;\mu)}{\|q_N\|_Y\|v_N\|_X}>0 \quad\forall\;\mu\in\cD.
\end{equation}
In case of $\varepsilon=0$, (\ref{eq:rb_scheme}) is then uniquely solvable for $(u^k_N(\mu),p^k_{N}(\mu)), k\in\bK$, if and only if the RB approximation spaces $X_N, Y_N$ are stable; in case of $\varepsilon>0$, corresponding to our comments on the truth problem in Remark~1.1, (\ref{eq:rb_scheme}) is uniquely solvable for $(u^{\varepsilon,k}_N(\mu),p^{\varepsilon,k}_N(\mu)), k\in\bK$, for any choice of $X_N, Y_N$ (see \cite{Brezzi:1991fk,Gerner:2012ag}). 

\subsection{{\em A Posteriori} Error Estimation} \label{ss:error_estimation}

We now develop upper bounds for the errors in our RB approximations that are rigorous, sharp, and computationally efficient. In this context, symmetric problems shall be discussed as a special case in which these bounds can be further sharpened.

In this section, we assume that the low-dimensional RB spaces $X_N, Y_N$ are constructed such that for any given $\mu\in\cD$, a solution $(u^{\varepsilon,k}_N(\mu),p^{\varepsilon,k}_N(\mu))\in X_N\times Y_N$, $k\in\bK$, to (\ref{eq:rb_scheme}) exists (see \S \ref{ss:Galerkin_projection}). 
For $\mu\in\cD$, we then consider the errors
\begin{align} \nonumber
e^u_N(\mu) &\equiv (e^{u,k}_N(\mu))_{k\in\bK}, \text{ where } e^{u,k}_N(\mu) \equiv u^{\varepsilon,k}(\mu) - u^{\varepsilon,k}_N(\mu) \in X, \;k\in\bK,\\ \label{eq:def_errors}
e^p_N(\mu) &\equiv (e^{p,k}_N(\mu))_{k\in\bK}, \text{ where } e^{p,k}_N(\mu) \equiv p^{\varepsilon,k}(\mu) -p^{\varepsilon,k}_N(\mu) \in Y,\; k\in\bK,\\ \nonumber
e^\varepsilon_N(\mu) &\equiv (e^{\varepsilon,k}_N(\mu))_{k\in\bK}, \text{ where } e^{\varepsilon,k}_N(\mu) \equiv (e^{u,k}_N(\mu),e^{p,k}_N(\mu))\in Z, \;k\in\bK,
\end{align}
in the RB approximations $(u^{\varepsilon,k}_N(\mu),p^{\varepsilon,k}_N(\mu)), k\in\bK$, with respect to the truth solution $(u^{\varepsilon,k}(\mu),p^{\varepsilon,k}(\mu)), k\in\bK$; we note that in particular $e^{u,0}_N(\mu)\equiv u^{\varepsilon,0}(\mu) - u^{\varepsilon,0}_N(\mu)=0$ from our initial conditions. 

To formulate our RB {\em a posteriori} error bounds, we rely on the residuals associated with the RB approximation $(u^{\varepsilon,k}_N(\mu),p^{\varepsilon,k}_N(\mu)), k\in\bK$,
\begin{align} \nonumber
r^{1,k}_{N}(\cdot;\mu) &\equiv f^k(\mu) - {\textstyle \frac{1}{\Delta t}}\,m(u^{\varepsilon,k}_{N}(\mu)-u^{\varepsilon,k-1}_{N}(\mu),v;\mu)\\ \label{eq:residual_1_k}
&\hspace{26ex} - a(u^{\varepsilon,k}_{N}(\mu),v;\mu) - b(v,p^{\varepsilon,k}_{N}(\mu);\mu) \in X',\\ \label{eq:residual_2_k}
r^{2,k}_{N}(\cdot;\mu) &\equiv g^k(\mu) - b(u^{\varepsilon,k}_{N}(\mu),q;\mu)+ \varepsilon\, c(p^{\varepsilon,k}_{N}(\mu),q;\mu) \in Y'
\end{align}
for $k\in\bK$ and $\mu\in\cD$.

In the following analysis, we distinguish between the cases $\varepsilon=0$ and $\varepsilon>0$.

\subsubsection{$\varepsilon=0$} \label{sss:errbnds}

We here derive rigorous upper bounds for the error $e^{u}_N(\mu)$ measured in the ``spatio-temporal'' energy norm
\begin{equation} \label{eq:energy_norm}
\|(v^j)_{j\in\bK}\|_{\ell^2(0,k;X)} \equiv \Bigg(\|v^k\|^2_\mu + \Delta t \sum_{j=1}^k \|v^j\|^2_{X,\mu}\Bigg)^{1/2},\quad (v^j)_{j\in\bK}\subseteq X, \quad k\in\bK.
\end{equation}
Our RB {\em a posteriori} error bounds shall be formulated in terms of the dual norms of the residuals (\ref{eq:residual_1_k}) and (\ref{eq:residual_2_k}), and (Online-)efficient lower and upper bounds to the truth continuity, coercivity, and inf-sup constants (\ref{eq:gamma_m}), (\ref{eq:gamma_a}), (\ref{eq:alpha_a}), and (\ref{eq:LBB}),
\begin{align} \label{eq:constants_LB_UB}
\setlength\arraycolsep{2pt}
\begin{array}{ccccc}
\gamma_m^\LB(\mu) & \leq & \gamma_m(\mu) & \leq & \gamma_m^\UB(\mu),\\[0.5ex]
\gamma_a^\LB(\mu) & \leq & \gamma_a(\mu) & \leq & \gamma_a^\UB(\mu),\\[0.5ex]
\alpha_a^\LB(\mu) & \leq &  \alpha_a(\mu) & \leq & \alpha_a^\UB(\mu),\\[0.5ex]
\beta^\LB(\mu) & \leq &  \beta(\mu) & \leq & \beta^\UB(\mu),
\end{array}
\quad\forall\;\mu\in\cD.
\end{align}

We can now state the following result.
\begin{proposition} \label{prpstn:errbnd}
For any given $\mu\in\cD$, $N\in\bNmax$, $k\in\bK$, and $\alpha_a^\LB(\mu)$, $\gamma^\UB_a(\mu)$, $\beta^\LB(\mu)$, $\gamma_m^\UB(\mu)$ satisfying (\ref{eq:constants_LB_UB}), we define 
\begin{multline} \label{eq:energy_errbnd}
\Delta^{k}_N(\mu) \equiv \Bigg[ \Delta t \sum_{j=1}^k \frac{\|r^{1,j}_N(\cdot;\mu)\|^2_{X'}}{{\alpha^\LB_a(\mu)}} + \frac{2}{\beta^\LB(\mu)}\Bigg(1+{\frac{\gamma^\UB_a(\mu)}{\alpha^\LB_a(\mu)}}\Bigg) \|r^{1,j}_N(\cdot;\mu)\|_{X'}\|r^{2,j}_N(\cdot;\mu)\|_{Y'}\\
 + \Bigg(\frac{\gamma^\UB_m(\mu)}{\Delta t} + \frac{(\gamma^\UB_a(\mu))^2}{\alpha^\LB_a(\mu)}\Bigg) \frac{\|r^{2,j}_N(\cdot;\mu)\|^2_{Y'}}{(\beta^\LB(\mu))^2}\Bigg]^{1/2}.
\end{multline}
Then, $\Delta^{k}_N(\mu)$ represents an upper bound for the error $e^{u}_N(\mu)$ measured in the ``spatio-temporal'' energy norm (\ref{eq:energy_norm}),
\begin{equation} \label{eq:rig_energy_errbnd}
\|e^{u}_N(\mu)\|_{\ell^2(0,k;X)} \leq \Delta^{k}_N(\mu) \quad\forall\; k\in\bK, \;\mu\in\cD, \;N\in\bNmax.
\end{equation}
\end{proposition}
\begin{proof}
Let $\mu$ be any parameter in $\cD$, $N\in\bNmax$, and $k\in\bK$. For clarity of exposition, we suppress the argument $\mu$ in this proof.

Take any $1\leq j\leq k$. From (\ref{eq:residual_1_k}), (\ref{eq:residual_2_k}), and (\ref{eq:truth_scheme}), the errors $e^{u,j}_N\in X$ and $e^{p,j}_N\in Y$ satisfy the equations
\begin{align} \label{eq:error_mom}
{\textstyle \frac{1}{\Delta t}}\, m(e^{u,j}_N-e^{u,j-1}_N,v) + a(e^{u,j}_N,v) + b(v,e^{p,j}_N) &= r^{1,j}_N(v) \quad\forall\;v\in X,\\ \label{eq:error_cont}
b(e^{u,j}_N,q) &= r^{2,j}_N(q) \quad\forall\;q\in Y.
\end{align}
By the LBB inf-sup condition (\ref{eq:LBB}) and (\ref{eq:error_mom}), we have
\begin{align} \nonumber
\beta \|e^{p,j}_N\|_Y &\leq \sup_{v\in X} \frac{b(v, e^{p,j}_N)}{\|v\|_X} = \sup_{v\in X} \frac{ r^{1,j}_N(v) - a(e^{u,j}_N,v)-\frac{1}{\Delta t} m(e^{u,j}_N-e^{u,j-1}_N,v)}{\|v\|_X}\\ \label{eq:bound_ep_j}
&\leq \|r^{1,j}_N\|_{X'} + {\gamma_a} \|e^{u,j}_N\|_{X} + \frac{\sqrt{\gamma_m}}{\Delta t}  \|e^{u,j}_N-e^{u,j-1}_N\|_\mu,
\end{align}
where the last inequality follows from the Cauchy--Schwarz inequality for the inner product $m(\cdot,\cdot)$, (\ref{eq:gamma_m}), and (\ref{eq:gamma_a}). We then set $v=e^{u,j}_N$, $q=e^{p,j}_N$ in (\ref{eq:error_mom}), (\ref{eq:error_cont}) and subtract the second from the first equation such that
\begin{align*}
{\textstyle \frac{1}{\Delta t}}\, m(e^{u,j}_N-e^{u,j-1}_N,e^{u,j}_N) + \|e^{u,j}_N\|^2_{X,\mu} & = r^{1,j}_N(e^{u,j}_N) - r^{2,j}_N(e^{p,j}_N) \\
&\leq \|r^{1,j}_N\|_{X'} \|e^{u,j}_N\|_X + \|r^{2,j}_N\|_{Y'}\|e^{p,j}_N\|_Y.
\end{align*}
Applying now (\ref{eq:bound_ep_j}) and (\ref{eq:alpha_a}) yields
\begin{multline*} 
{\textstyle \frac{1}{\Delta t}}\,m(e^{u,j}_N-e^{u,j-1}_N,e^{u,j}_N) + \|e^{u,j}_N\|^2_{X,\mu}\\ 
 \leq \frac{1}{\beta} \|r^{1,j}_N\|_{X'}\|r^{2,j}_N\|_{Y'}+ \bigg(\|r^{1,j}_N\|_{X'} + \frac{{\gamma_a}}{\beta}\|r^{2,j}_N\|_{Y'}\bigg) \frac{\|e^{u,j}_N\|_{X,\mu}}{\sqrt{\alpha_a}} \\
+ \frac{1}{\Delta t}\frac{\sqrt{\gamma_m}}{\beta} \|r^{2,j}_N\|_{Y'} \|e^{u,j}_N-e^{u,j-1}_N\|_\mu,
\end{multline*}
which can be further bounded from Young's inequality by
\begin{multline*} 
 \leq \frac{1}{\beta} \|r^{1,j}_N\|_{X'}\|r^{2,j}_N\|_{Y'}+ \frac{1}{2\alpha_a}\bigg(\|r^{1,j}_N\|_{X'} + \frac{{\gamma_a}}{\beta}\|r^{2,j}_N\|_{Y'}\bigg)^2 + \frac{1}{2} \|e^{u,j}_N\|^2_{X,\mu} \\
+\frac{1}{2\Delta t}\frac{\gamma_m}{\beta^2} \|r^{2,j}_N\|^2_{Y'} + \frac{1}{2\Delta t} \|e^{u,j}_N-e^{u,j-1}_N\|^2_\mu.
\end{multline*}
Rearranging terms, the inequality now reads
\begin{multline*}
\frac{1}{\Delta t} \left(\|e^{u,j}_N\|^2_\mu - \|e^{u,j-1}_N\|^2_\mu \right) + \|e^{u,j}_N\|^2_{X,\mu}\\
\leq \frac{\|r^{1,j}_N\|^2_{X'} }{\alpha_a}+ \frac{2}{\beta} \bigg(1+ {\frac{\gamma_a}{\alpha_a}}\bigg) \|r^{1,j}_N\|_{X'}\|r^{2,j}_N\|_{Y'} + \bigg(\frac{\gamma_m}{\Delta t} + \frac{\gamma_a^2}{\alpha_a} \bigg)\frac{\|r^{2,j}_N\|^2_{Y'}}{\beta^2},
\end{multline*}
and the result follows from applying the sum $\sum_{j=1}^k$, $e^{u,0}_N=0$, and (\ref{eq:constants_LB_UB}).
\end{proof}

In the special case of a symmetric problem, the error bounds given in Proposition~\ref{prpstn:errbnd} can be improved (see also \cite{Gerner:2012fk}). We may then derive the following result.

\begin{proposition} \label{prpstn:errbnd_sym}
Let $a(\cdot,\cdot;\mu)$ be symmetric for all $\mu\in\cD$. For any given $\mu\in\cD$, $N\in\bNmax$, $k\in\bK$, and $\alpha_a^\LB(\mu)$, $\gamma^\UB_a(\mu)$, $\beta^\LB(\mu)$, $\gamma_m^\UB(\mu)$ satisfying (\ref{eq:constants_LB_UB}), we define 
\begin{multline} \label{eq:energy_errbnd_sym}
\Delta^{{\rm sym},k}_N(\mu) \equiv \Bigg[ \Delta t \sum_{j=1}^k \frac{\|r^{1,j}_N(\cdot;\mu)\|^2_{X'}}{{\alpha^\LB_a(\mu)}}\\
 + \frac{2}{\beta^\LB(\mu)}\Bigg(1+\sqrt{\frac{\gamma^\UB_a(\mu)}{\alpha^\LB_a(\mu)}}\Bigg) \|r^{1,j}_N(\cdot;\mu)\|_{X'}\|r^{2,j}_N(\cdot;\mu)\|_{Y'}\\
 + \Bigg(\frac{\gamma^\UB_m(\mu)}{\Delta t} + {\gamma^\UB_a(\mu)}\Bigg) \frac{\|r^{2,j}_N(\cdot;\mu)\|^2_{Y'}}{(\beta^\LB(\mu))^2}\Bigg]^{1/2}.
\end{multline}
Then, $\Delta^{{\rm sym},k}_N(\mu)$ represents an upper bound for the error $e^{u}_N(\mu)$ measured in the ``spatio-temporal'' energy norm (\ref{eq:energy_norm}),
\begin{equation} \label{eq:rig_energy_errbnd_sym}
\|e^{u}_N(\mu)\|_{\ell^2(0,k;X)} \leq \Delta^{{\rm sym},k}_N(\mu) \quad\forall\; k\in\bK, \;\mu\in\cD, \;N\in\bNmax.
\end{equation}
\end{proposition}
\begin{proof}
Following the lines of the previous proof, we may now apply the Cauchy--Schwarz inequality for the inner product $a(\cdot,\cdot)$ to obtain
\begin{equation*} 
\beta \|e^{p,j}_N\|_Y \leq \|r^{1,j}_N\|_{X'} + \sqrt{\gamma_a} \|e^{u,j}_N\|_{X,\mu} + \frac{\sqrt{\gamma_m}}{\Delta t}  \|e^{u,j}_N-e^{u,j-1}_N\|_\mu,
\end{equation*}
instead of (\ref{eq:bound_ep_j}). Proceeding as before, this yields
\begin{multline*} 
{\textstyle \frac{1}{\Delta t}}\,m(e^{u,j}_N-e^{u,j-1}_N,e^{u,j}_N) + \|e^{u,j}_N\|^2_{X,\mu}\\ 
 \leq \frac{1}{\beta} \|r^{1,j}_N\|_{X'}\|r^{2,j}_N\|_{Y'}+ \bigg(\frac{\|r^{1,j}_N\|_{X'}}{\sqrt{\alpha_a}} + \frac{\sqrt{\gamma_a}}{\beta}\|r^{2,j}_N\|_{Y'}\bigg) \|e^{u,j}_N\|_{X,\mu} \\
+ \frac{1}{\Delta t}\frac{\sqrt{\gamma_m}}{\beta} \|r^{2,j}_N\|_{Y'} \|e^{u,j}_N-e^{u,j-1}_N\|_\mu,
\end{multline*}
and the statement again follows from applying Young's inequality, the sum $\sum_{j=1}^k$, {$e^{u,0}_N=0$}, and (\ref{eq:constants_LB_UB}).
\end{proof}

\subsubsection{$\varepsilon>0$} \label{sss:errbnds_penalty}

We here derive rigorous upper bounds for the error $e^{\varepsilon}_N(\mu)$ measured in the ``spatio-temporal'' energy norm
\begin{align}
 \label{eq:energy_norm_penalty}
\|(v^j,q^j)_{j\in\bK}\|_{\ell^2(0,k;Z)} \equiv \Bigg(\|v^k\|^2_\mu + \Delta t \sum_{j=1}^k \|v^j\|^2_{X,\mu} + \varepsilon\, \|q^j\|^2_{Y,\mu}\Bigg)^{1/2},
\end{align}
where $(v^j,q^j)_{j\in\bK}\subseteq Z$, $k\in\bK$. 

In addition to the dual norms of the residuals (\ref{eq:residual_1_k}) and (\ref{eq:residual_2_k}), we here also rely on (Online-)efficient lower (and upper) bounds to the truth coercivity constants (\ref{eq:alpha_a}) and (\ref{eq:alpha_c}),
\begin{align} \label{eq:alpha_LB}
\setlength\arraycolsep{2pt}
\begin{array}{ccccc}
\alpha_a^\LB(\mu) & \leq &  \alpha_a(\mu) & \leq & \alpha_a^\UB(\mu),\\[0.5ex]
\alpha_c^\LB(\mu) & \leq &  \alpha_c(\mu) & \leq & \alpha_c^\UB(\mu),
\end{array}
\quad\forall\;\mu\in\cD,
\end{align}
to formulate our RB {\em a posteriori} error bounds.

To demonstrate the differences to the case where $\varepsilon=0$, we recall the following result together with its proof (see \cite{Gerner:2012uq}).

\begin{proposition} \label{prpstn:errbnds_penalty}
For any given $\mu\in\cD$, $N\in\bNmax$, $k\in\bK$, and $\alpha_a^\LB(\mu)$, $\alpha^\LB_c(\mu)$ satisfying (\ref{eq:alpha_LB}), we define 
\begin{equation} \label{eq:energy_errbnd_penalty}
\Delta^{\varepsilon,k}_N(\mu) \equiv \Bigg(\Delta t \sum_{j=1}^k \frac{\|r^{1,j}_{N}(\cdot;\mu)\|^2_{X'}}{\alpha^\LB_a(\mu)} + \frac{\|r^{2,j}_{N}(\cdot;\mu)\|^2_{Y'}}{\varepsilon \alpha^\LB_c(\mu)}\Bigg)^{1/2}.
\end{equation}
Then, $\Delta^{\varepsilon,k}_N(\mu)$ represents an upper bound for the error $e^{\varepsilon}_N(\mu)$ measured in the ``spatio-temporal'' energy norm (\ref{eq:energy_norm_penalty}),
\begin{equation} \label{eq:rig_energy_errbnd_penalty}
\|e^{\varepsilon}_N(\mu)\|_{\ell^2(0,k;Z)} \leq \Delta^{\varepsilon,k}_N(\mu) \quad\forall\; k\in\bK, \;\mu\in\cD, \;N\in\bNmax.
\end{equation}
\end{proposition}
\begin{proof}
Let $\mu$ be any parameter in $\cD$, $N\in\bNmax$, and $k\in\bK$. For clarity of exposition, we suppress the argument $\mu$ in this proof.

Take any $1\leq j\leq k$. From (\ref{eq:residual_1_k}), (\ref{eq:residual_2_k}), and (\ref{eq:truth_scheme}), the errors $e^{u,j}_N\in X$ and $e^{p,j}_N\in Y$ satisfy the equations
\begin{align*}
{\textstyle \frac{1}{\Delta t}}\, m(e^{u,j}_N-e^{u,j-1}_N,v) + a(e^{u,j}_{N},v) + b(v,e^{p,j}_{N}) &= r^{1,j}_{N}(v)\quad \forall\;v\in X,\\
b(e^{u,j}_{N},q) -\varepsilon\, c(e^{p,j}_{N},q) & = r^{2,j}_{N}(q) \quad\forall\; q\in Y.
\end{align*}
Setting here $v=e^{u,j}_{N}$, $q=e^{p,j}_{N}$ and subtracting the second from the first equation, we obtain
\begin{multline}\label{eq:inst_penalty_inequ_1}
{\textstyle\frac{1}{\Delta t}}\, m(e^{u,j}_N-e^{u,j-1}_N,e^{u,j}_{N}) + \|e^{u,j}_{N}\|^2_{X,\mu} + \varepsilon\, \|e^{p,j}_{N}\|^2_{Y,\mu} = r^{1,j}_{N}(e^{u,j}_{N}) - r^{2,j}_{N}(e^{p,j}_{N})\\
\leq \|r^{1,j}_{N}\|_{X'} \|e^{u,j}_{N}\|_X + \|r^{2,j}_{N}\|_{Y'} \|e^{p,j}_{N}\|_Y.
\end{multline}
On the right-hand side, we now use (\ref{eq:alpha_a}), (\ref{eq:alpha_c}), and Young's inequality so that
\begin{multline*}
\|r^{1,j}_{N}\|_{X'} \|e^{u,j}_{N}\|_X + \|r^{2,j}_{N}\|_{Y'} \|e^{p,j}_{N}\|_Y\\
 \leq \frac{\|r^{1,j}_{N}\|_{X'}}{\sqrt{\alpha_a}}\|e^{u,j}_{N}\|_{X,\mu} + \frac{\|r^{2,j}_{N}\|_{Y'}}{\sqrt{\alpha_c}}\|e^{p,j}_{N}\|_{Y,\mu}\\
 \leq \frac{1}{2}\left(\frac{\|r^{1,j}_{N}\|^2_{X'}}{\alpha_a} + \|e^{u,j}_{N}\|^2_{X,\mu} + \frac{\|r^{2,j}_{N}\|^2_{Y'}}{\varepsilon \alpha_c} + \varepsilon\,\|e^{p,j}_{N}\|^2_{Y,\mu}\right);
\end{multline*}
on the left-hand side, we use the Cauchy--Schwarz inequality for the inner product $m(\cdot,\cdot)$ followed by Young's inequality so that
\begin{align*}
m(e^{u,j}_N-e^{u,j-1}_N, e^{u,j}_N) \geq \|e^{u,j}_N\|^2_\mu - \|e^{u,j-1}_N\|_\mu\|e^{u,j}_N\|_\mu \geq \frac{1}{2} \left(\|e^{u,j}_N\|^2_\mu-\|e^{u,j-1}_N\|^2_\mu\right).
\end{align*}
Rearranging terms, the inequality (\ref{eq:inst_penalty_inequ_1}) finally reads
\begin{equation*}
{\frac{1}{\Delta t}}\Big(\|e^{u,j}_N\|^2_\mu - \|e^{u,j-1}_N\|^2_\mu\Big) + \|e^{u,j}_{N}\|^2_{X,\mu} + \varepsilon\, \|e^{p,j}_{N}\|^2_{Y,\mu} \leq \frac{\|r^{1,j}_{N}\|^2_{X'}}{\alpha_a} + \frac{\|r^{2,j}_{N}\|^2_{Y'}}{\varepsilon \alpha_c},
\end{equation*}
and the statement follows from applying the sum $\sum_{j=1}^k$, $e^{u,0}_N=0$, and (\ref{eq:alpha_LB}).
\end{proof}

Through the introduction of the penalty term, we thus obtain {\em a posteriori} error bounds that do not depend on inf-sup constants. However, we note that they depend on the penalty parameter $\varepsilon$: As $\varepsilon$ decreases and we approach the nonperturbed problem, (\ref{eq:energy_errbnd_penalty}) suggests a growth by an order of $O(\textstyle \frac{1}{\sqrt{\varepsilon}})$.

\subsection{Offline-Online Computational Procedure} \label{ss:offline_online}

The efficiency of the RB method relies on an Offline-Online computational decomposition strategy. As it is by now standard, we shall only provide a brief summary at this point and refer the reader to, e.g., \cite{Grepl:2005kx,Rozza:2008fv} for further details.
The procedure requires that all involved operators can be affinely expanded with respect to the parameter $\mu$. All $\mu$-independent quantities are formed and stored within a computationally expensive Offline stage, which is performed only once and whose cost depends on the large finite element dimension $\cN$. For any given parameter $\mu\in\cD$, the RB approximation $(u^{\varepsilon,k}_N(\mu),p^{\varepsilon,k}_N(\mu)), k\in\bK$, is then computed within a highly efficient Online stage; the cost does not depend on $\cN$ but only on the much smaller dimension of the RB approximation space.
 The computation of the {\em a posteriori} error bounds consists of two components: the calculation of the residual dual norms $\|r^{1,k}_{N}(\cdot;\mu)\|_{X'}$, $\|r^{2,k}_{N}(\cdot;\mu)\|_{Y'}$, $k\in\bK$, and the calculation of the required lower and upper bounds (\ref{eq:constants_LB_UB}) and (\ref{eq:alpha_LB}), respectively, to the involved constants. The former is again an application of now standard RB techniques that can be found in \cite{Grepl:2005kx,Rozza:2008fv}. The latter is achieved by a successive constraint method (SCM) as proposed in \cite{DBP-Huynh:2007ez}; we also refer the reader to \cite{Gerner:2011fk} for details in our saddle point context.

\subsection{Construction of Reduced Basis Approximation Spaces} \label{ss:construction_RB_spaces}

We now turn to the construction of the RB approximation spaces $X_N, Y_N$, $N\in\bNmax$. The low-dimensional spaces $X_N, Y_N$ are constructed by exploiting the parametric structure of the problem: According to the so-called Lagrange approach, basis functions are essentially given by truth solutions associated with several chosen parameter snapshots. 
However, in our time-dependent setting, $X_N$ and $Y_N$ not only have to appropriately represent the submanifold induced by the parametric dependence but also need to capture the causality associated with evolution in time to provide accurate approximations $(u^{\varepsilon,k}_N(\mu), p^{\varepsilon,k}_N(\mu))$ for $(u^{\varepsilon,k}(\mu),p^{\varepsilon,k}(\mu))$, $k\in\bK$, for any parameter query. Keeping computational cost to a minimum, we aim to achieve this with as few basis functions as possible. 

The POD greedy procedure represents an adaptive sampling process for parabolic problems that properly accounts for temporal and parametric causality: It combines the proper orthogonal decomposition (POD) method in $k$ (see \cite{Kunisch:2001fk,Kunisch:2002bh}) with the greedy procedure in $\mu$ (see \cite{dahmen10:_greedy,Buffa:kx} and  \cite{Gerner:2011fk}). To begin with, we briefly recall the optimality property of the POD as described in \cite{Kunisch:2001fk,Kunisch:2002bh}. For a given finite set $\mathcal{X}_\mathcal{I}\equiv \{\chi_1,\ldots,\chi_{\mathcal{I}}\}\subseteq X$  and $M_X\leq {\rm dim}({\rm span}(\mathcal{X}_{\mathcal{I}}))$, the POD basis of rank $M_X$ consists of $M_X$ $(\cdot,\cdot)_X$-orthonormal basis functions that approximate $\mathcal{X}_{\mathcal{I}}$ best in the sense that
\begin{equation*}
{\rm span}({\rm POD}_X(\mathcal{X}_{\mathcal{I}},M_X)) = \arg \inf_{\substack{\mathcal{X}\,\subseteq\, {\rm span}(\mathcal{X}_{\mathcal{I}})\\ {\rm dim}(\mathcal{X})=M_X}} \left(\frac{1}{\mathcal{I}}\sum_{i=1}^{\mathcal{I}} \inf_{\chi \in \mathcal{X}} \|\chi_i-\chi\|^2_X\right)^{1/2};
\end{equation*}
analogously, we denote by ${\rm POD}_Y(\mathcal{Y}_{\mathcal{I}},M_Y)$ the POD basis of rank $M_Y$ for a finite set \mbox{$\mathcal{Y}_{\mathcal{I}}\subseteq Y$},  $M_Y\leq {\rm dim}({\rm span}(\mathcal{Y}_{\mathcal{I}}))$.
Assuming that we are given a current pair $(X_{N-1}, Y_{N-1})$ of RB approximation spaces, the POD greedy algorithm now proceeds as follows: In compliance with the greedy approach, it detects the parameter $\mu_N$ for which the (Online-)efficient RB error bound attains its maximum over an exhaustive sample $\Sigma\subset \cD$. For a prescribed $\Delta N\in\bK$, we then compute the POD bases of rank $\Delta N$ associated with the truth solutions $u^{\varepsilon,k}(\mu_N)$ and $p^{\varepsilon,k}(\mu_N)$, $k\in\bK$; more specifically, we compute ${\rm POD}_X(E^u,\Delta N)$ and ${\rm POD}_Y(E^p,\Delta N)$ for
\begin{align*}
E^u &\equiv \{\, u^{\varepsilon,k}(\mu_N) - \Pi_{X_{N-1}} u^{\varepsilon,k}(\mu_N) \mid k\in\bK \,\},\\
E^p &\equiv \{\, p^{\varepsilon,k}(\mu_N) - \Pi_{Y_{N-1}} p^{\varepsilon,k}(\mu_N) \mid k\in\bK \,\},
\end{align*}
where $\Pi_{X_{N-1}}$ and $\Pi_{Y_{N-1}}$ refer to the $(\cdot,\cdot)_X$- and $(\cdot,\cdot)_Y$-orthogonal projections on the current RB approximation spaces $X_{N-1}$ and $Y_{N-1}$, respectively. Finally, the $\Delta N$ POD basis functions are appended to $X_{N-1}$ and $Y_{N-1}$, and we obtain a subsequent pair $(X_N,Y_N)$. This process is then repeated until a prescribed error tolerance is satisfied. We refer the reader to \cite{Grepl:2005kx,Haasdonk:2011ffk,haasdonk08:_reduc} for a detailed discussion of the POD greedy procedure, and to \cite{knezevic09:_reduc_basis_approx_poster_error,knezevic10:_certif_reduc_basis_method_fokker} for an application to the Boussinesq and Fokker--Planck equations.

\begin{algorithm}[!bp]                      
\caption{Adaptive Sampling Procedure for $\varepsilon=0$}          
\label{alg:modified_POD}                           
\begin{algorithmic}[1]                    
\STATE Choose $\Sigma \subset \cD$, $\delta_{\rm tol}, \delta^\beta_{\rm tol}\in(0,1)$,  $\Delta N\in\bK$, and $\mu_1 \in\Sigma$
\STATE Set $N\leftarrow 0$, $\cD_N\leftarrow \{ \}$, $\cD'\leftarrow \{ \}$, $N_Y \leftarrow 0$, $Y_N \leftarrow \{ \}$, $N_X \leftarrow 0$, $X_N\leftarrow \{ \}$
\REPEAT
\STATE $N \leftarrow N+1$, $\cD_{N} \leftarrow \cD_{N-1} \cup \{\mu_N\}$
\STATE $E^p = \{\, p^{k}(\mu_N) - \Pi_{Y_{N-1}} p^{k}(\mu_N) \mid  k\in\bK \,\}$
\STATE $N_Y \leftarrow N_Y+\Delta N$, $Y_N \leftarrow Y_{N-1} \oplus {\rm span}({\rm POD}_Y(E^p,\Delta N))$
\IF {$\mu_N\notin \cD'$,}
\STATE $E^u = \{\, u^{k}(\mu_N) - \Pi_{X_{N-1}} u^{k}(\mu_N) \mid  k\in\bK \,\}$
\STATE $N_X \leftarrow N_X+\Delta N$, $X_N \leftarrow X_{N-1} \oplus {\rm span}({\rm POD}_X(E^u,\Delta N))$
\ENDIF
\WHILE {({\bf true})}
\FORALL {$\mu \in \Sigma$}
\STATE Compute $(u^k_N(\mu),p^k_N(\mu))$, $k\in\bK$, $\Delta_N(\mu)$, and 
\STATE $\hat d^\beta_N(\mu) \equiv \max \left\{\,\frac{\beta^\UB(\mu) - \beta_N(\mu)}{\beta^\UB(\mu)},0\,\right\}$ (cf.~(\ref{eq:dist_betaN}))
\ENDFOR
\STATE $ \mu'_{N} \equiv \arg \max_{\mu \in \Sigma}\; \Delta_N(\mu)$, $ \mu^* \equiv \arg \max_{\mu\in\Sigma} \hat d^\beta_N(\mu)$
\IF {$\hat d^\beta_N(\mu^*) < \delta^\beta_{\rm tol}$,}
\STATE $\mu_{N+1} \equiv \mu'_N$
\STATE {\bf break}
\ENDIF
\IF {$ \min_{\mu\in \cD'\cup\cD_N} \frac{|\mu'_N-\mu|}{|\mu|} \geq 0.1\%$,}
\STATE $\cD'\leftarrow \cD' \cup \{\mu'_N\}$
\STATE $E^u = \{\, u^{k}(\mu'_N) - \Pi_{X_{N}} u^{k}(\mu'_N) \mid  k\in\bK \,\}$
\STATE $N_X\leftarrow N_X+\Delta N$, $X_N \leftarrow X_{N} \oplus {\rm span}({\rm POD}_X(E^u,\Delta N))$
\ELSE
\STATE $N_X\leftarrow N_X+1$, $X_N \leftarrow X_{N}\oplus {\rm span}\{\,T_{\mu^*} \varrho_N(\mu^*) \,\}$ (see (2.28), (2.36) in  \cite{Gerner:2011fk})
\ENDIF
\ENDWHILE
\UNTIL {$\Delta_N(\mu_{N+1}) < \delta_{\rm tol}$}
\STATE $N_{\rm max} \leftarrow N$
\end{algorithmic}
\end{algorithm}

\begin{algorithm}[!bp]                      
\caption{Adaptive Sampling Procedure for $\varepsilon>0$}          
\label{alg:modified_POD_penalty}                           
\begin{algorithmic}[1]                    
\STATE Choose $\Sigma \subset \cD$, $\delta_{\rm tol}\in(0,1)$, $\delta^\kappa_{\rm tol}>0$, $\Delta N\in\bK$, and $\mu_1 \in\Sigma$
\STATE Set $N\leftarrow 0$, $\cD_N\leftarrow \{ \}$, $\cD'\leftarrow \{ \}$, $N_Y \leftarrow 0$, $Y_N \leftarrow \{ \}$, $N_X \leftarrow 0$, $X_N\leftarrow \{ \}$
\REPEAT
\STATE $N \leftarrow N+1$, $\cD_{N} \leftarrow \cD_{N-1} \cup \{\mu_N\}$
\STATE $E^p = \{\, p^{\varepsilon,k}(\mu_N) - \Pi_{Y_{N-1}} p^{\varepsilon,k}(\mu_N) \mid  k\in\bK \,\}$
\STATE $N_Y \leftarrow N_Y+\Delta N$, $Y_N \leftarrow Y_{N-1} \oplus {\rm span}({\rm POD}_Y(E^p,\Delta N))$
\IF {$\mu_N\notin \cD'$,}
\STATE $E^u = \{\, u^{\varepsilon,k}(\mu_N) - \Pi_{X_{N-1}} u^{\varepsilon,k}(\mu_N) \mid  k\in\bK \,\}$
\STATE $N_X \leftarrow N_X+\Delta N$, $X_N \leftarrow X_{N-1} \oplus {\rm span}({\rm POD}_X(E^u,\Delta N))$
\ENDIF
\WHILE {({\bf true})}
\FORALL {$\mu \in \Sigma$}
\STATE Compute $(u^{\varepsilon,k}_N(\mu),p^{\varepsilon,k}_N(\mu))$, $k\in\bK$, $\Delta_N(\mu)$, and $\kappa^\varepsilon_N(\mu)$ (see (\ref{eq:kappa_N}))
\ENDFOR
\STATE $ \mu'_{N} \equiv \arg \max_{\mu \in \Sigma}\; \Delta_N(\mu)$, $ \mu^* \equiv \arg \max_{\mu\in\Sigma} \kappa^\varepsilon_N(\mu)$
\IF {$\kappa^\varepsilon_N(\mu^*) < \delta^\kappa_{\rm tol}$,}
\STATE $\mu_{N+1} \equiv \mu'_N$
\STATE {\bf break}
\ENDIF
\IF {$ \min_{\mu\in \cD'\cup\cD_N} \frac{|\mu'_N-\mu|}{|\mu|} \geq 0.1\%$,}
\STATE $\cD'\leftarrow \cD' \cup \{\mu'_N\}$
\STATE $E^u = \{\, u^{\varepsilon,k}(\mu'_N) - \Pi_{X_N} u^{\varepsilon,k}(\mu'_N) \mid  k\in\bK \,\}$
\STATE $N_X\leftarrow N_X+\Delta N$, $X_N \leftarrow X_{N} \oplus {\rm span}({\rm POD}_X(E^u,\Delta N))$
\ELSE
\STATE $N_X\leftarrow N_X+1$, $X_N \leftarrow X_{N}\oplus {\rm span}\{\,T_{\mu^*} \varrho_N(\mu^*) \,\}$ (see (2.28), (2.36) in  \cite{Gerner:2011fk})
\ENDIF
\ENDWHILE
\UNTIL {$\Delta_N(\mu_{N+1}) < \delta_{\rm tol}$}
\STATE $N_{\rm max} \leftarrow N$
\end{algorithmic}
\end{algorithm}

For our saddle point problems, we now couple the above procedure with stabilization techniques developed in \cite{Gerner:2011fk}; here (see also \cite{Gerner:2012fk}), best convergence results were achieved by Algorithm~3 that aims to stabilize $X_N, Y_N$ adaptively through an enrichment of the primal RB approximation space with additional truth solutions. According to these observations, we now apply the sampling procedures presented in Algorithm~\ref{alg:modified_POD} and Algorithm~\ref{alg:modified_POD_penalty}. In case of $\varepsilon=0$, we use the distance $d^\beta_N(\mu)$ (see~\cite{Gerner:2011fk}),
\begin{equation}\label{eq:dist_betaN}
d^\beta_N(\mu) \equiv \max\left\{\frac{\beta(\mu)-\beta_N(\mu)}{\beta(\mu)},0\right\}, \quad\mu\in\cD,
\end{equation}
of the inf-sup constants $\beta_N(\mu)$ to the truth inf-sup constants $\beta(\mu)$ as an indicator whether a current pair of RB approximation spaces needs to be stabilized; the exact procedure is given in Algorithm~\ref{alg:modified_POD}. In case of $\varepsilon>0$, numerical results in \cite{Gerner:2012ag} showed that the inf-sup constants $\beta_N(\mu)$ may not be appropriate indicators for an ill-conditioned system but an adaptive sampling process should be based rather on the condition number $\kappa^\varepsilon_N(\mu)$,
\begin{equation} \label{eq:kappa_N}
\kappa^\varepsilon_N(\mu) \equiv \frac{\sigma^{\varepsilon,\rm max}_N(\mu)}{\sigma^{\varepsilon,\rm min}_N(\mu)},\quad \varepsilon>0,\quad\mu\in\cD,\; N\in\bNmax;
\end{equation}
here, $\sigma^{\varepsilon,\rm max}_N(\mu)$ and $\sigma^{\varepsilon,\rm min}_N(\mu)$ denote the maximum and minimum singular values of the corresponding RB system matrix, respectively. Algorithm~\ref{alg:modified_POD_penalty} now presents a possibility how this could be realized.

\section{Model Problem} \label{s:model_problem}

We consider a Stokes flow in a two-dimensional microchannel with an obstacle as introduced in \cite{veroy10:_stokes_penalty}; evolution in time is now induced by a time-dependent velocity profile on the inflow boundary.

Let $\mu$ be any parameter in $\cD$. For the physical domain $\tilde\Omega$ and a given time interval $[0,T]$, $T>0$, we now seek to find the (inhomogeneous) velocity $\tilde u_{\rm e, inh}:\tilde\Omega\times (0,T)\to\mathbb{R}^2$ and the pressure $\tilde p_{\rm e}:\tilde\Omega\times (0,T)\to\mathbb{R}$ satisfying
\begin{align} \label{eq:strong_inh_mom}
\frac{\partial \tilde u_{\rm e, inh}}{\partial t} - \tilde\Delta \tilde u_{\rm e, inh} + \tilde\nabla \tilde p_{\rm e} &= 0 \quad\text{in }\tilde\Omega\times (0,T),\\ \label{eq:strong_inh_cont}
\tilde \nabla\cdot \tilde u_{\rm e,inh} &= 0 \quad\text{in }\tilde\Omega\times (0,T),
\end{align}
subject to initial conditions $\tilde u_{\rm e, inh}(\cdot,0) = 0$ and with boundary conditions
\begin{align} \label{eq:inh_bc} \begin{split}
&\tilde u_{\rm e, inh}(\tilde x,t) = H(t) h(\tilde x) \quad\text{on } \Gamma_{\rm in} \times (0,T),\quad \tilde{u}_{\rm e, inh}=0 \quad\text{on }\tilde\Gamma_0 \times (0,T), \\
&\hspace{20ex}\frac{\partial \tilde u_{\rm e, inh}}{\partial \tilde n} = \tilde p_{\rm e} \tilde n \quad\text{on }\Gamma_{\rm out} \times (0,T);
\end{split} \end{align}
here, $\tilde\Delta$ and $\tilde\nabla$ denote the Laplacian and gradient operator over the physical domain $\tilde\Omega$, $\tilde n$ is the unit outward normal, $h:\mathbb{R}^2\to\mathbb{R}^2$ is given by $h(x) \equiv (4x_2(1-x_2),0)$ for all $x=(x_1,x_2)\in \mathbb{R}^2$, and we choose $H:[0,T]\to\mathbb{R}$ with $H(t) \equiv t (\sin(2\pi t)+1)$ for all $t\in[0,T]$. 
According to the setting introduced in \cite{veroy10:_stokes_penalty}, we also consider the following perturbation of the problem (\ref{eq:strong_inh_mom})--(\ref{eq:inh_bc}): For a sufficiently small $\varepsilon>0$, we introduce a penalty term into the continuity equation (\ref{eq:strong_inh_cont}) such that
\begin{equation}\label{eq:strong_inh_cont_penalty}
\tilde \nabla\cdot \tilde u^\varepsilon_{\rm e,inh} = -\varepsilon\, p^\varepsilon_{\rm e} \quad\text{in }\tilde\Omega\times (0,T).
\end{equation}

We now follow the steps discussed in \cite{veroy10:_stokes_penalty}: We choose the lifting function $\tilde u^H_{\rm L} \equiv H \tilde u_{\rm L}$ where $\tilde u_{\rm L}$ is defined as in \cite{veroy10:_stokes_penalty}, and transform the problem statement for the homogeneous velocity $\tilde u^\varepsilon_{\rm e} \equiv \tilde u^\varepsilon_{\rm e, inh}-\tilde u^H_{\rm L}$ to an equivalent problem posed over the reference domain $\Omega$. Furthermore, as required for the time-discrete setting introduced in \S \ref{s:general_problem_statement}, we divide the time interval $[0,T]$ into $K$ subintervals of equal length $\Delta t \equiv T/K$, and consider a backward Euler method for time integration. The problems (\ref{eq:strong_inh_mom})--(\ref{eq:inh_bc}) and (\ref{eq:strong_inh_mom}), (\ref{eq:strong_inh_cont_penalty}), (\ref{eq:inh_bc}) may thus be written as a parametrized saddle point problem of the form (\ref{eq:exact_scheme}). 
Here, for any $\mu\in\cD$, the bilinear forms $a(\cdot,\cdot;\mu)$, $b(\cdot,\cdot;\mu)$, and $c(\cdot,\cdot;\mu)$ are given as in \cite{veroy10:_stokes_penalty}; accordingly, the bilinear form $m(\cdot,\cdot;\mu):\Xe\times \Xe \to\mathbb{R}$ represents the $L^2$-inner product for vector functions over the physical domain $\tilde\Omega$ formulated on the reference domain $\Omega$,
\begin{equation*}
m(u,v;\mu) = \sum_{s=1}^S \frac{1}{|{\rm det}(A^s(\mu))|} \int_{\Omega^s} u\cdot v\,dx \quad\forall\; u,v\in \Xe,
\end{equation*}
and the linear functionals $f(\cdot;\mu)$ and $g(\cdot;\mu)$ are given by
\begin{align*}
f(v,t;\mu) &= f(v,t) = -H'(t)\int_{\Omega_{\rm L}} u_{\rm L}\cdot v\,dx - H(t) \int_{\Omega_{\rm L}}\frac{\partial u_{{\rm L}i}}{\partial x_j} \frac{\partial v_i}{\partial x_j}\,dx,\\
g(q,t;\mu) &= g(q,t) = H(t) \int_{\Omega_{\rm L}} q \frac{\partial u_{{\rm L}i}}{\partial x_i}\,dx
\end{align*}
for all $v\in\Xe$, $q\in\Ye$, $t\in[0,T]$. We recall that the bilinear forms $a(\cdot,\cdot;\mu)$, $b(\cdot,\cdot;\mu)$, and $c(\cdot,\cdot;\mu)$ then satisfy the assumptions (\ref{eq:gamma_a_e})--(\ref{eq:brezzi_infsup_e}), (\ref{eq:gamma_c_e}), and (\ref{eq:alpha_c_e}). For all $\mu\in\cD$, $m(\cdot,\cdot;\mu)$ defines an inner product on $\Xe$ such that (\ref{eq:pos_def_m_e}) holds true; moreover, there exists a constant $C^{\rm e}(\mu)>0$ from the Poincar\'{e} inequality (see, e.g., \cite{Quarteroni:2008fk}) such that  
\begin{equation*} 
 m(v,v;\mu) \leq C^{\rm e}(\mu) \, a(v,v;\mu) \quad\forall\; v\in \Xe \quad\forall\;\mu\in\cD,
\end{equation*}
and thus (\ref{eq:gamma_m_e}) is satisfied with
\begin{equation*}
\gamma^{\rm e}_m(\mu) \equiv \sup_{u\in X_{\rm e}}\sup_{v\in \Xe} \frac{m(u,v;\mu)}{\|u\|_\Xe \|v\|_{X_{\rm e}}} \leq C^{\rm e}(\mu) \,\gamma^{\rm e}_a(\mu) < \infty \quad\forall\;\mu\in\cD.
\end{equation*}

Choosing the truth approximation spaces $X$ and $Y$ as the standard conforming $\mathbb{P}_2$-$\mathbb{P}_1$ Taylor--Hood finite element approximation subspaces \cite{Taylor:1973fk} over the regular triangulation $\cT_\Omega$, we ensure that also (\ref{eq:LBB}) is satisfied (see, e.g., \cite{Brezzi:1991fk,Ern:2004jl,Raviart:1986uq,Quarteroni:2008fk}) and therefore recover the situation described in \S \ref{ss:truth}.

\section{Numerical Results} \label{s:numerical_results}

We now apply the RB methodology developed in \S \ref{s:RB_method} to our model problem introduced in \S \ref{s:model_problem}. We set $T=1$ and consider a constant time step size $\Delta t$ corresponding to $K=100$ time levels. The truth discretization is based on a fine mesh with a total of $\cN =$ 72,076 velocity and pressure degrees of freedom.
In this section, all numerical results are attained using the open source software \texttt{rbOOmit}~\cite{Knezevic:2011fk}, an implementation of the RB framework within the C++ parallel finite element library \texttt{libMesh}~\cite{libMeshPaper}.

\subsection{$\varepsilon=0$}

We first turn to the coercivity, continuity, and inf-sup constants required for our RB procedure. We obtain (Online-)efficient lower and upper bounds to $\alpha_a(\mu)$, $\gamma_a(\mu)$, and $\beta(\mu)$ by using the SCM (see \S \ref{ss:offline_online}) with the configurations specified in \cite{Gerner:2011fk}. To estimate the continuity constants $\gamma_m(\mu)$, we apply the method for $M_\alpha=\infty$, $M_+=0$, an exhaustive sample $\Xi\subset\cD$ of size $|\Xi|$ = 4,225, and the SCM tolerance $\epsilon=0.01$ (see \cite{DBP-Huynh:2007ez}). We then obtain accurate (Online-)efficient lower and upper bounds $\gamma_m^\LB(\mu)$ and $\gamma^\UB_m(\mu)$ with $K_{\rm max}=5$.

We now turn to the RB approximation. To build our low-dimensional RB approximation spaces $X_N, Y_N$, $N\in\bNmax$, we apply the POD greedy procedure described in Algorithm~\ref{alg:modified_POD} (see \S \ref{ss:construction_RB_spaces}). The sampling process is based on an exhaustive random sample $\Sigma\subset\cD$ of size $|\Sigma|=$ 4,900, $\Delta N=2$, and $\delta^\beta_{\rm tol}=0.1$; since our Stokes model problem is clearly symmetric, we here in particular use the relative RB {\em a posteriori} error bound $\Delta_N(\mu)\equiv \Delta^{{\rm sym},K}_N(\mu) / \|(u^j_N(\mu))_{j\in\bK}\|_{\ell^2(0,K;X)}$ (see (\ref{eq:energy_norm}), (\ref{eq:energy_errbnd_sym})). 

\begin{figure}[htp]
\centering
\epsfig{file=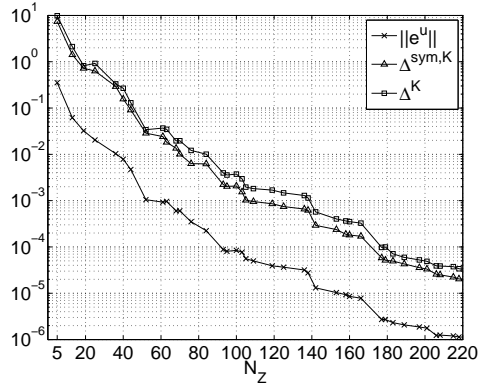,scale=0.35}\\[-3ex]
\caption{Maximum error $\|e^{u}_N(\mu)\|_{\ell^2(0,K;X)}$ (see (\ref{eq:def_errors}), (\ref{eq:energy_norm})) and maximum error bounds $\Delta^{{\rm sym},K}_N(\mu)$ and $\Delta^K_N(\mu)$ (see (\ref{eq:energy_errbnd_sym}) and (\ref{eq:energy_errbnd})) normalized with respect to $\|(u^j(\mu))_{j\in\bK}\|_{\ell^2(0,K;X)}$ shown as functions of $N_Z$; the maximum is taken over 25 parameter values.}
\label{fig:err_errbnd_K}
\end{figure}

\begin{figure}[htp]
\centerline{\footnotesize $N=13$ $(N_Z = 76)$ \hspace{13ex} $N=20$ $(N_Z = 109)$ \hspace{13ex} $N=41$ $(N_Z = 226)$}
\epsfig{file=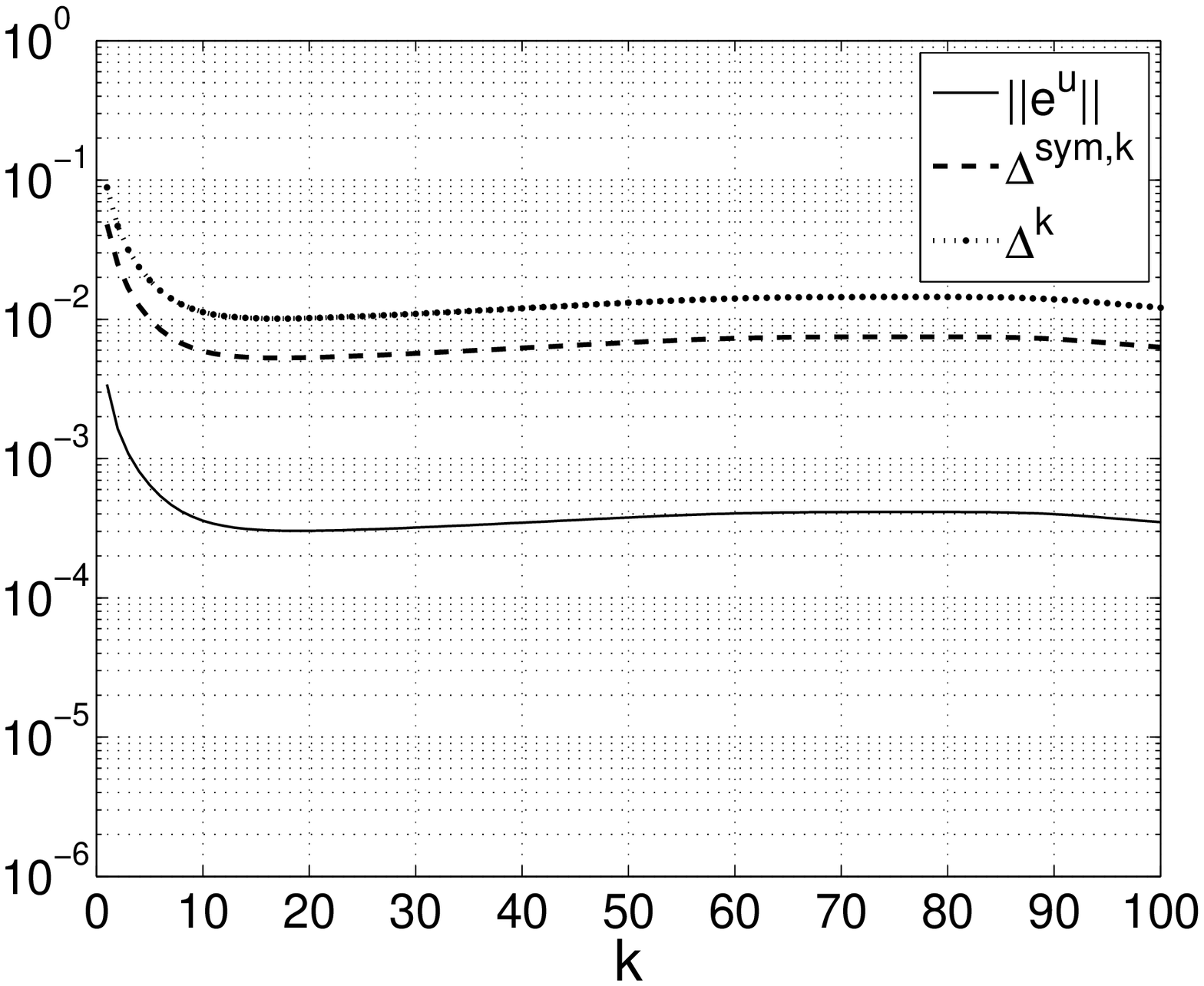,scale=0.24}\hfill
\epsfig{file=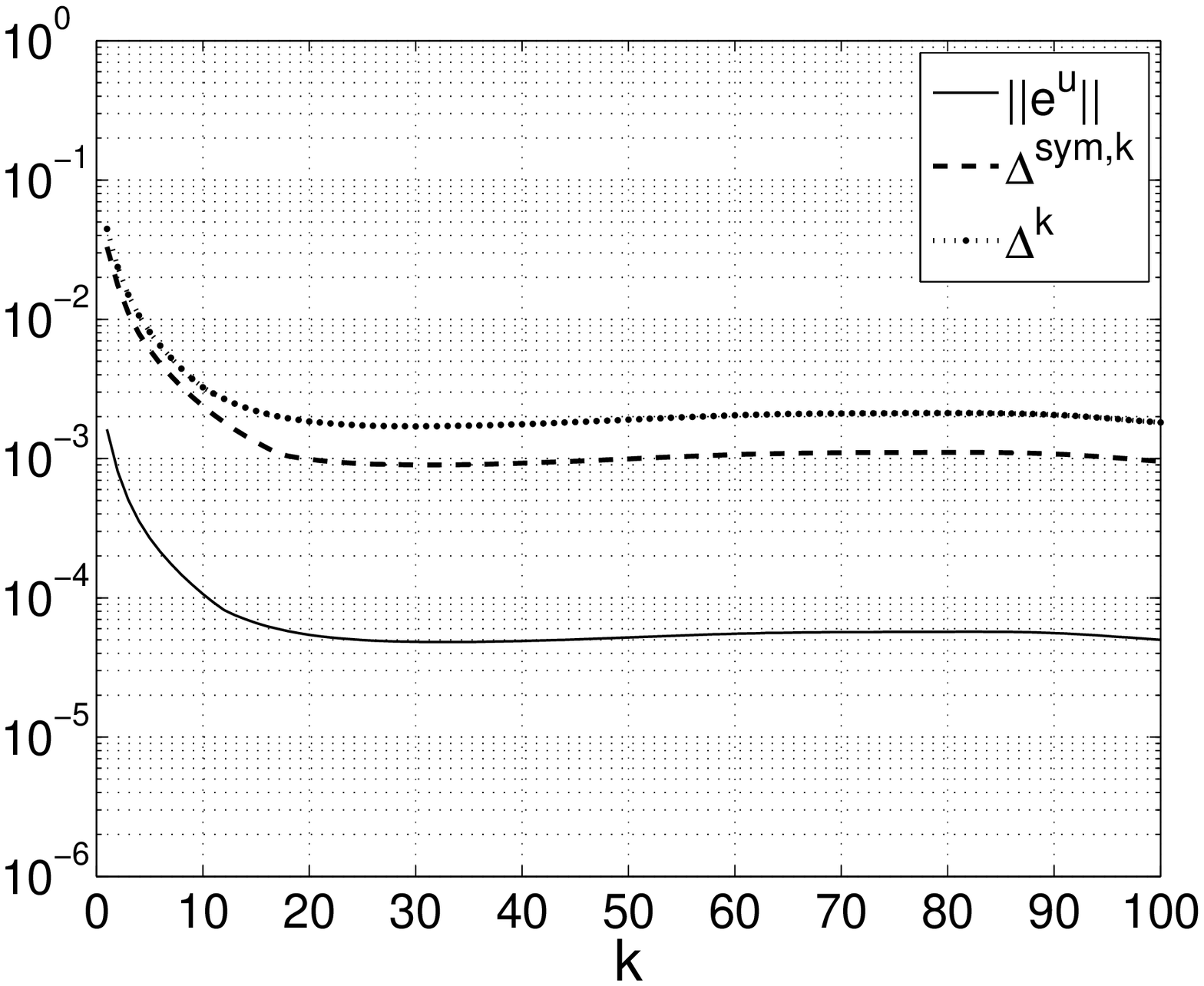,scale=0.24}\hfill
\epsfig{file=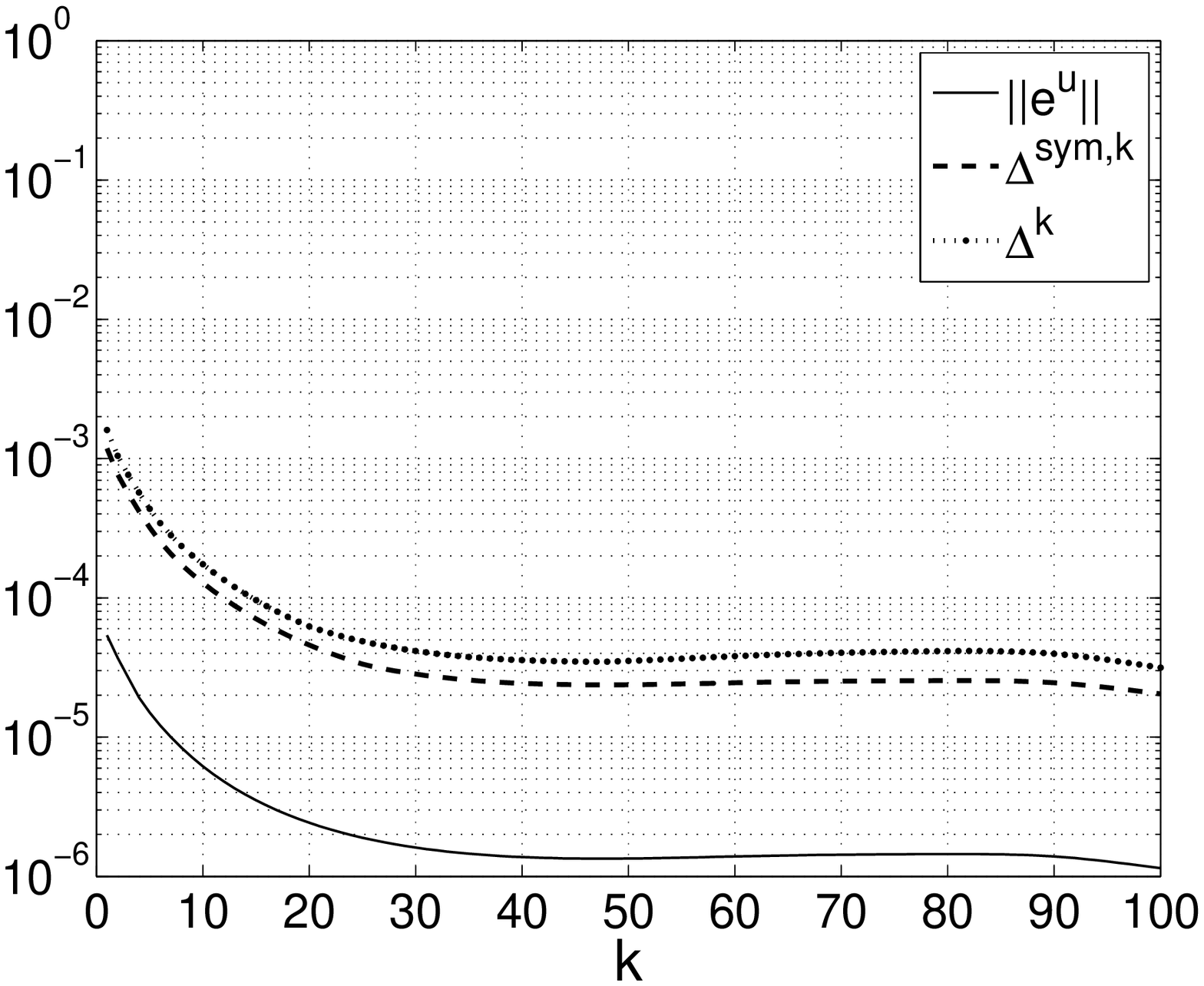,scale=0.24}\hfill \\[-5ex]
\caption{Maximum error $\|e^{u}_N(\mu)\|_{\ell^2(0,k;X)}$ (see (\ref{eq:def_errors}), (\ref{eq:energy_norm})) and maximum error bounds $\Delta^{{\rm sym},k}_N(\mu)$ and $\Delta^k_N(\mu)$ (see (\ref{eq:energy_errbnd_sym}) and (\ref{eq:energy_errbnd})) normalized with respect to $\|(u^j(\mu))_{j\in\bK}\|_{\ell^2(0,k;X)}$ shown as functions of $k\in\bK$ for several values of $N$; the maximum is taken over 25 parameter values.}
\label{fig:err_errbnd_N}
\end{figure}

\begin{table}[htp]
\centering \footnotesize
\begin{tabular}{@{}c|c|c|c|c|c|c|c@{}}
\multicolumn{8}{l}{(a) Effectivities $\eta^{{\rm sym},k}_N(\mu)$ associated with $\Delta^{{\rm sym},k}_N(\mu)$}\\[1ex]
\hline
$N$ & $N_Z$ & $k=10$ & $k=20$ & $k=40$ & $k=60$ & $k=80$ & $k=100$ \\
\hline
$5$ & $36$ & $28.90$ & $30.38$ & $31.30$ & $31.61$ & $31.63$ & $31.13$\\
$10$ & $63$ & $30.41$ & $31.35$ & $32.20$ & $32.37$ & $32.16$ & $31.91$\\
$15$ & $93$ & $25.83$ & $28.05$ & $29.39$ & $29.51$ & $29.44$ & $29.38$\\
$20$ & $109$ & $23.31$ & $24.25$ & $26.79$ & $27.23$ & $27.08$ & $27.21$\\
$25$ & $142$ & $25.29$ & $28.15$ & $29.54$ & $29.73$ & $29.66$ & $29.64$\\
$30$ & $177$ & $26.28$ & $26.05$ & $28.77$ & $30.58$ & $30.70$ & $30.60$\\
$35$ & $201$ & $24.77$ & $24.86$ & $26.68$ & $27.36$ & $27.51$ & $27.81$\\
$40$ & $222$ & $24.18$ & $23.96$ & $24.19$ & $25.03$ & $25.51$ & $25.54$\\
\hline
\end{tabular}\\[1ex]
\begin{tabular}{@{}c|c|c|c|c|c|c|c@{}}
\multicolumn{8}{l}{(b) Effectivities $\eta^{k}_N(\mu)$ associated with $\Delta^k_N(\mu)$}\\[1ex]
\hline
$N$ & $N_Z$ & $k=10$ & $k=20$ & $k=40$ & $k=60$ & $k=80$ & $k=100$ \\
\hline
$5$ & $36$ & $39.58$ & $41.51$ & $42.76$ & $43.20$ & $43.22$ & $42.53$\\
$10$ & $63$ & $39.90$ & $41.66$ & $42.82$ & $43.21$ & $43.00$ & $42.90$\\
$15$ & $93$ & $38.58$ & $43.21$ & $44.64$ & $45.13$ & $45.14$ & $45.07$\\
$20$ & $109$ & $32.59$ & $34.10$ & $37.11$ & $37.60$ & $37.17$ & $37.19$\\
$25$ & $142$ & $35.52$ & $39.26$ & $42.93$ & $42.73$ & $42.63$ & $43.55$\\
$30$ & $177$ & $34.31$ & $34.27$ & $36.41$ & $37.31$ & $39.67$ & $40.64$\\
$35$ & $201$ & $32.86$ & $33.59$ & $36.52$ & $37.39$ & $36.65$ & $37.54$\\
$40$ & $222$ & $33.67$ & $33.76$ & $35.25$ & $35.42$ & $34.78$ & $35.17$\\
\hline
\end{tabular}\\[1ex]
\caption{Maximum effectivities (a)~$\eta^{{\rm sym},k}_N(\mu)\equiv \Delta^{{\rm sym},k}_N(\mu)/\|e^{u}_N(\mu)\|_{\ell^2(0,k;X)}$ (see (\ref{eq:rig_energy_errbnd_sym})) and (b)~$\eta^{k}_N(\mu)\equiv \Delta^{k}_N(\mu)/\|e^{u}_N(\mu)\|_{\ell^2(0,k;X)}$ (see (\ref{eq:rig_energy_errbnd})) for several values of $k\in\bK$ and $N$; the maximum is taken over 25 parameter values.}
\label{tbl:effectivities}
\end{table}

Figure~\ref{fig:err_errbnd_K} now shows the maximum error $\|e^{u}_N(\mu)\|_{\ell^2(0,K;X)}$ (see (\ref{eq:def_errors})) in the RB velocity approximations and associated error bounds $\Delta^{{\rm sym},K}_N(\mu)$ and $\Delta^{K}_N(\mu)$ (see (\ref{eq:energy_errbnd})) as functions of the dimension $N_Z$; Figure~\ref{fig:err_errbnd_N} presents the maximum error $\|e^{u}_N(\mu)\|_{\ell^2(0,k;X)}$ and associated error bounds $\Delta^{{\rm sym},k}_N(\mu)$, $\Delta^k_N(\mu)$ as functions of $k\in\bK$ for several values of $N_Z$. 
First, we observe that the RB error and error bounds are roughly uniform in time (see Fig.~\ref{fig:err_errbnd_N}) and decrease rapidly as $N_Z$ increases (see Fig.~\ref{fig:err_errbnd_K}). We obtain stable, rapidly convergent RB approximations, and rigorous {\em a posteriori} error bounds that reflect the behavior of the error very accurately. Second, the error bounds are tight. To quantify this statement, we present in Table~\ref{tbl:effectivities} maximum effectivities associated with $\Delta^{{\rm sym},k}_N(\mu)$ and $\Delta^k_N(\mu)$ for several values of $k$ and $N$. We notice that their values remain more or less constant with $k$. Moreover, as in the stationary case (see \cite{Gerner:2012fk}), we benefit from exploiting the symmetry of the problem: Effectivities range from 33 to 45 in case of $\Delta^k_N(\mu)$ (see Table~\ref{tbl:effectivities}(b)) and improve in case of $\Delta^{{\rm sym},k}_N(\mu)$ by roughly 10 (see Table~\ref{tbl:effectivities}(a)). 
We emphasize at this point that the error bound formulations in (\ref{eq:energy_errbnd}) and (\ref{eq:energy_errbnd_sym}) in fact suggest a growth in time. In practice, this behavior seems rather weak (see Table~\ref{tbl:effectivities}) but may be investigated in greater detail within future work.

We now discuss the Online computation times for the proposed method. For comparison, once the $\mu$-independent parts in the affine expansions of the involved operators have been formed (see \S \ref{ss:offline_online}), direct computation of the truth approximation $(u^k(\mu),p^k(\mu)), k\in\bK$, (i.e., assembly and solution of (\ref{eq:truth_scheme})) requires roughly 30 seconds on a 2.66 GHz Intel Core 2 Duo processor. We initially take a total RB dimension of $N_Z=226$. Once the database has been loaded, the Online calculation of $(u^k_N(\mu), p^k_N(\mu)), k\in\bK$, (i.e., assembly and solution of (\ref{eq:rb_scheme})) and $\Delta^{{\rm sym},k}_N(\mu), k\in\bK$, for any new value of $\mu\in\cD$ takes on average 27.97 and 80.76 milliseconds, respectively, which is in total roughly 270 times faster than direct computation of the truth approximation. Thus, even for this large value of $N_Z$, we obtain significant Online savings. In practice, however, we quite often need not take such a large value of $N_Z$; our rigorous and inexpensive error bounds $\Delta^{{\rm sym},k}_N(\mu)$, $k\in\bK$, allow us to choose the RB dimension just large enough to obtain a desired accuracy. To achieve a prescribed accuracy of at least $1\%$ (resp., $0.1\%$) in the RB approximations $u^k_N(\mu), k\in\bK$, we need $N_Z = 76$ (resp., $N_Z = 109$) (see Fig.~\ref{fig:err_errbnd_K}). Again, once the database has been loaded, the Online calculation of $(u^k_N(\mu), p^k_N(\mu)), k\in\bK$, and $\Delta^{{\rm sym},k}_N(\mu), k\in\bK$, for any new value of $\mu\in\cD$ then takes on average 4.41 (resp., 7.62) and 24.47 (resp., 33.75) milliseconds, respectively, which is in total roughly 1,000 times (resp., 700 times) faster than direct computation of the truth approximation.

\subsection{$\varepsilon>0$}

Again, the SCM (see \S \ref{ss:offline_online}) enables the (Online-)efficient estimation of the coercivity  constants $\alpha_a(\mu)$ and $\alpha_c(\mu)$; as we here use the same configurations, we refer the reader to \cite{Gerner:2012ag} for details in this context. 

To build our low-dimensional RB approximation spaces $X_N, Y_N$, $N\in\bNmax$, we apply the POD greedy procedure described in Algorithm~\ref{alg:modified_POD_penalty} (see \S \ref{ss:construction_RB_spaces}). The sampling process is based on an exhaustive random sample $\Sigma\subset\cD$ of size $|\Sigma|=$ 4,900, $\Delta N=2$, $\delta^\kappa_{\rm tol}=10^3$, and the relative RB {\em a posteriori} error bound $\Delta_N(\mu)=\Delta^{\varepsilon,K}_N(\mu)/\|(u^{\varepsilon,j}_N(\mu))_{j\in\bK}\|_{\ell^2(0,K;Z)}$ (see (\ref{eq:energy_norm_penalty}), (\ref{eq:energy_errbnd_penalty})). 

\begin{figure}[htp]
\centerline{\footnotesize \hspace{4ex} $\varepsilon=10^{-2}$ \hspace{35ex} $\varepsilon = 10^{-3}$}
\centering
\epsfig{file=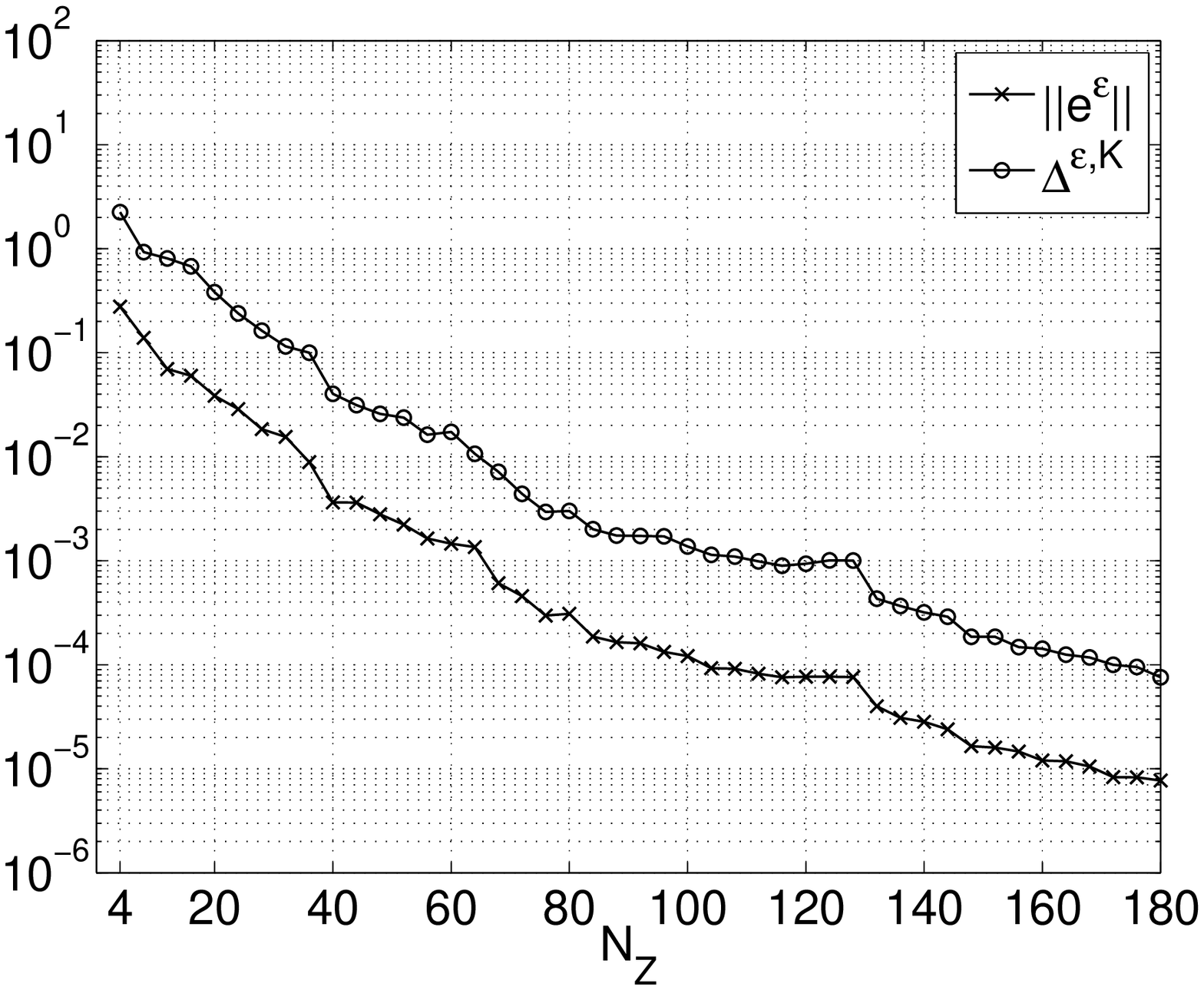,scale=0.3}\quad
\epsfig{file=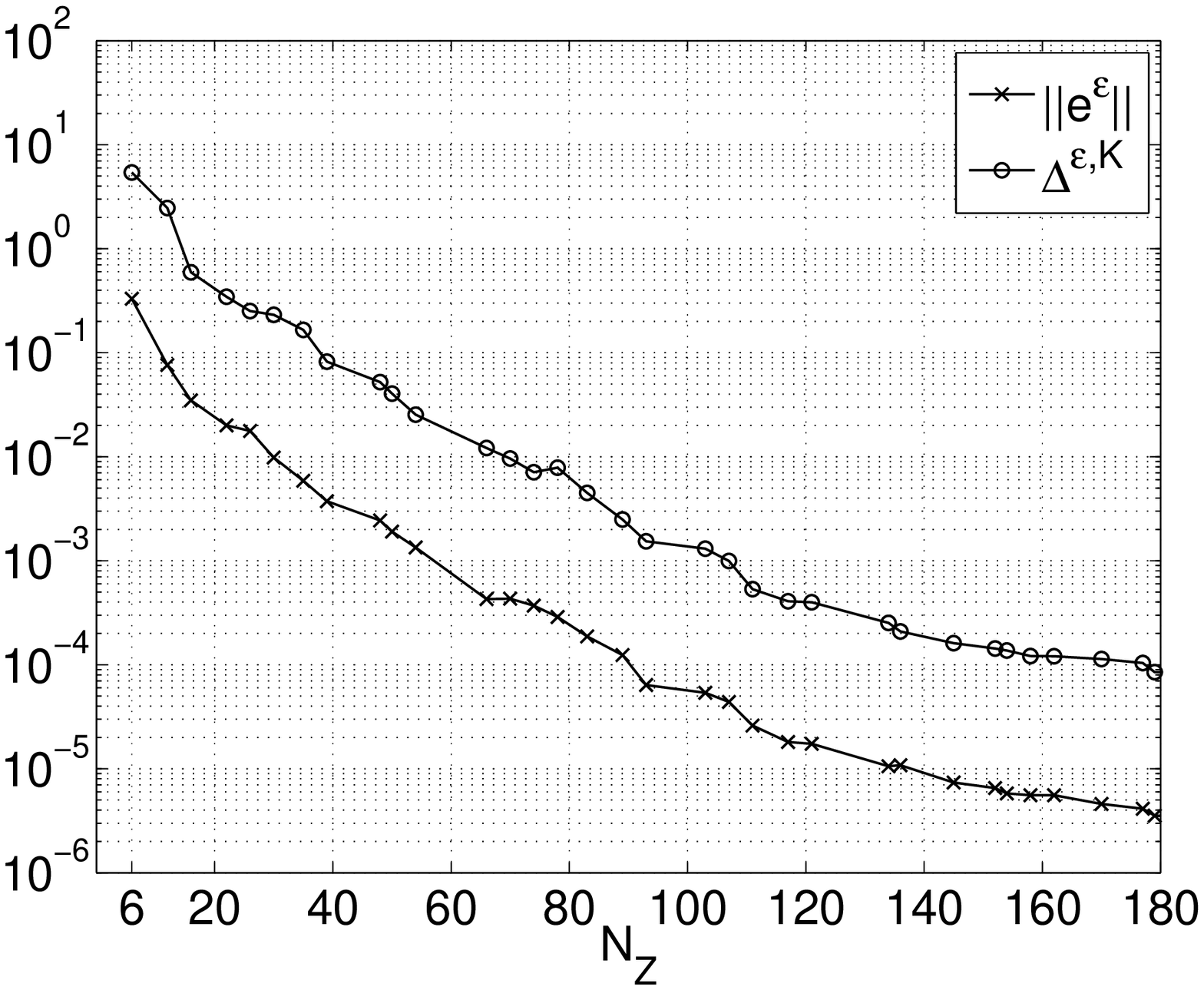,scale= 0.3}
\centerline{\footnotesize \hspace{4ex} $\varepsilon=10^{-4}$ \hspace{35ex} $\varepsilon = 10^{-5}$}
\epsfig{file=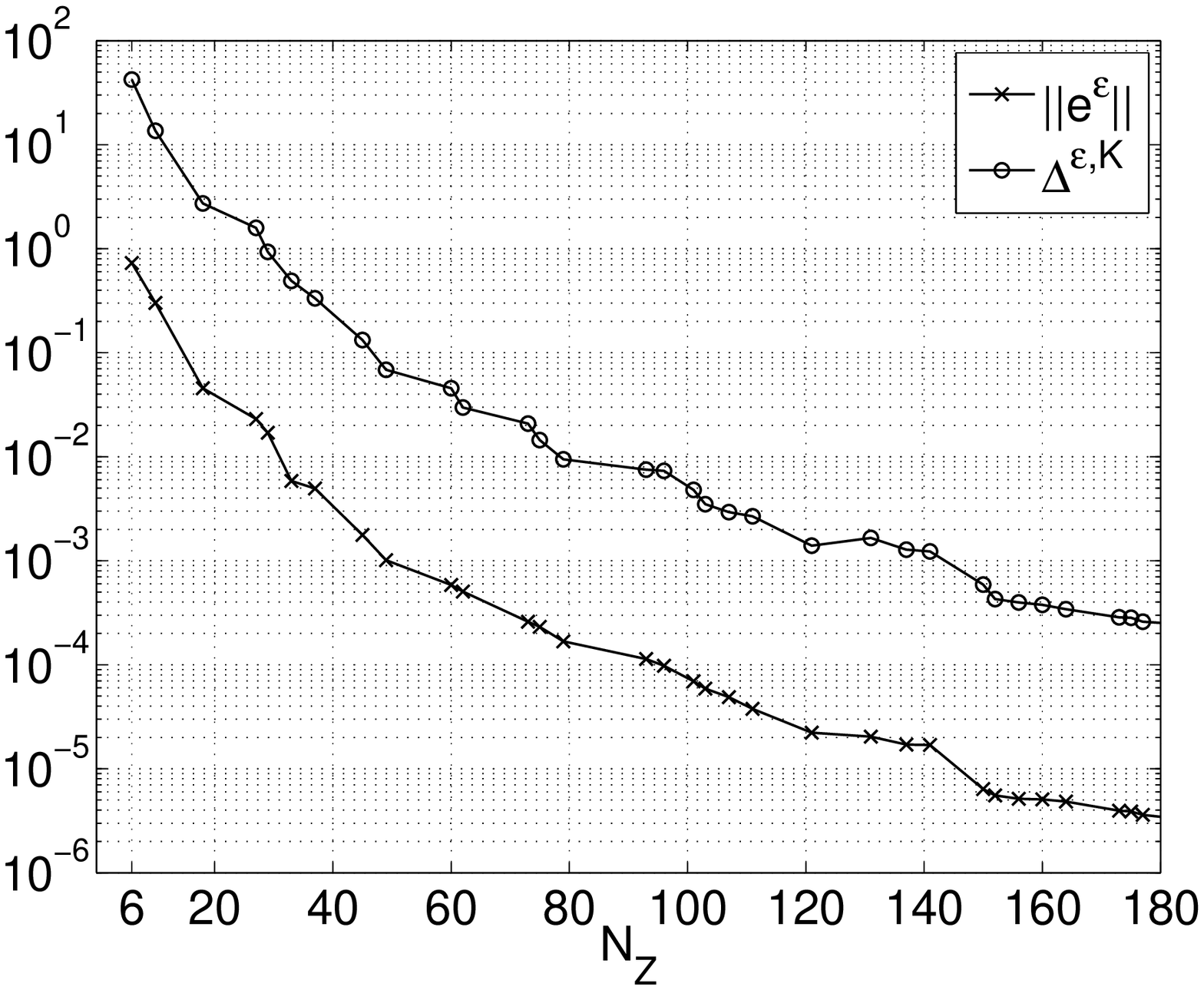,scale=0.3}\quad
\epsfig{file=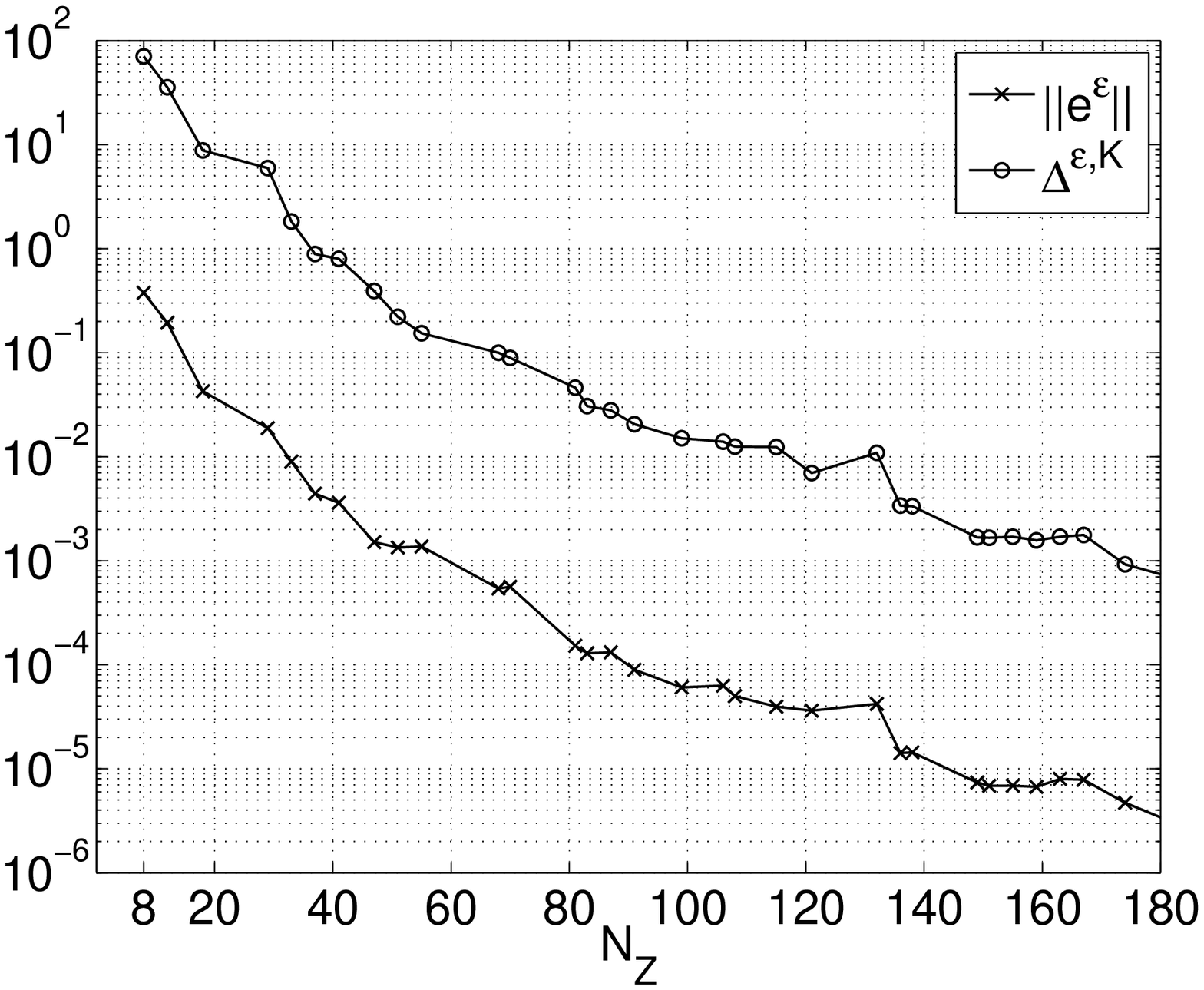,scale= 0.3}\\[-3ex]
\caption{Maximum error $\|e^{\varepsilon}_N(\mu)\|_{\ell^2(0,K;Z)}$ (see (\ref{eq:def_errors}), (\ref{eq:energy_norm_penalty})) and maximum error bound $\Delta^{\varepsilon,K}_N(\mu)$ (see (\ref{eq:energy_errbnd_penalty})) normalized with respect to $\|(u^{\varepsilon,j}(\mu),p^{\varepsilon,j}(\mu))_{j\in\bK}\|_{\ell^2(0,K;Z)}$ shown as functions of $N_Z$ for different values of $\varepsilon$; the maximum is taken over 25 parameter values.}
\label{fig:errbnd_alg6_penalty}
\end{figure}

Figure~\ref{fig:errbnd_alg6_penalty} now shows the maximum error $\|e^\varepsilon_N(\mu)\|_{\ell^2(0,K;Z)}$ (see (\ref{eq:def_errors})) in the RB velocity and pressure approximations together with the associated error bound $\Delta^{\varepsilon,K}_N(\mu)$ as functions of the dimension $N_Z$ for different values of $\varepsilon$. Figure~\ref{fig:err_errbnd_N_penalty} then presents the maximum error $\|e^\varepsilon_N(\mu)\|_{\ell^2(0,k;Z)}$ and associated error bound $\Delta^{\varepsilon,k}_N(\mu)$ as functions of $k\in\bK$ for several values of $N$; note that the latter are chosen as  the values for which the error bounds $\Delta^{\varepsilon,K}_N(\mu)$ guarantee a prescribed accuracy of at least $1\%$ and $0.1\%$ in the RB approximations.
First, we again observe that the RB error and error bounds are roughly uniform in time (see Fig.~\ref{fig:err_errbnd_N_penalty}) and decrease rapidly as $N_Z$ increases (see Fig.~\ref{fig:errbnd_alg6_penalty}). We obtain stable RB approximations whose rapid convergence is not affected by the penalty parameter, and {\em a posteriori} error bounds that are meaningful and rigorous. Second, using the condition numbers $\kappa^\varepsilon_N(\mu)$ as an indicator for an ill-conditioned system, Algorithm~\ref{alg:modified_POD_penalty} guarantees stability by properly accounting for the effects of the penalty term: For $\varepsilon=10^{-2}$, the sampling process recognizes that the RB approximation spaces $X_N,Y_N$ do not have to be stabilized to provide accurate approximations; taking smaller values of $\varepsilon$ and thus approaching the nonpenalized problem, an additional enrichment of the RB approximation space for the velocity becomes more and more necessary.
Third, we see that the error bounds are tight for $\varepsilon=10^{-2}$ but become less sharp as we decrease $\varepsilon$ and our perturbed truth approximation becomes more accurate. However, effectivities exhibit a similar $O\big(\frac{1}{\sqrt{\varepsilon}}\big)$-dependence on the penalty parameter as observed in the stationary case (see \cite{veroy10:_stokes_penalty}) and remain reasonably small for relatively small values of $\varepsilon$. To further quantify this statement, we present in Table~\ref{tbl:effectivities_penalty} the effectivities associated with $\Delta^{\varepsilon,k}_N(\mu)$ for different values of $k$, $N$, and $\varepsilon$. We note that their values are fairly constant with $k$ and $N$  and confirm the $O\big(\frac{1}{\sqrt{\varepsilon}}\big)$-dependence indicated by (\ref{eq:energy_errbnd_penalty}) as well as Fig.~\ref{fig:errbnd_alg6_penalty} and Fig.~\ref{fig:err_errbnd_N_penalty}. The effects of the penalty parameter on the effectivities are thus relatively benign and we obtain useful bounds for reasonably small values of $\varepsilon$.

\begin{figure}[!tp]
\centerline{
\begin{minipage}[c]{0.15\textwidth}
\footnotesize (a)~$\varepsilon = 10^{-2}$
\end{minipage} \quad
\begin{minipage}[c]{0.35\textwidth}
\centerline{\footnotesize $N=17$ ($N_Z=68$)}
\epsfig{file=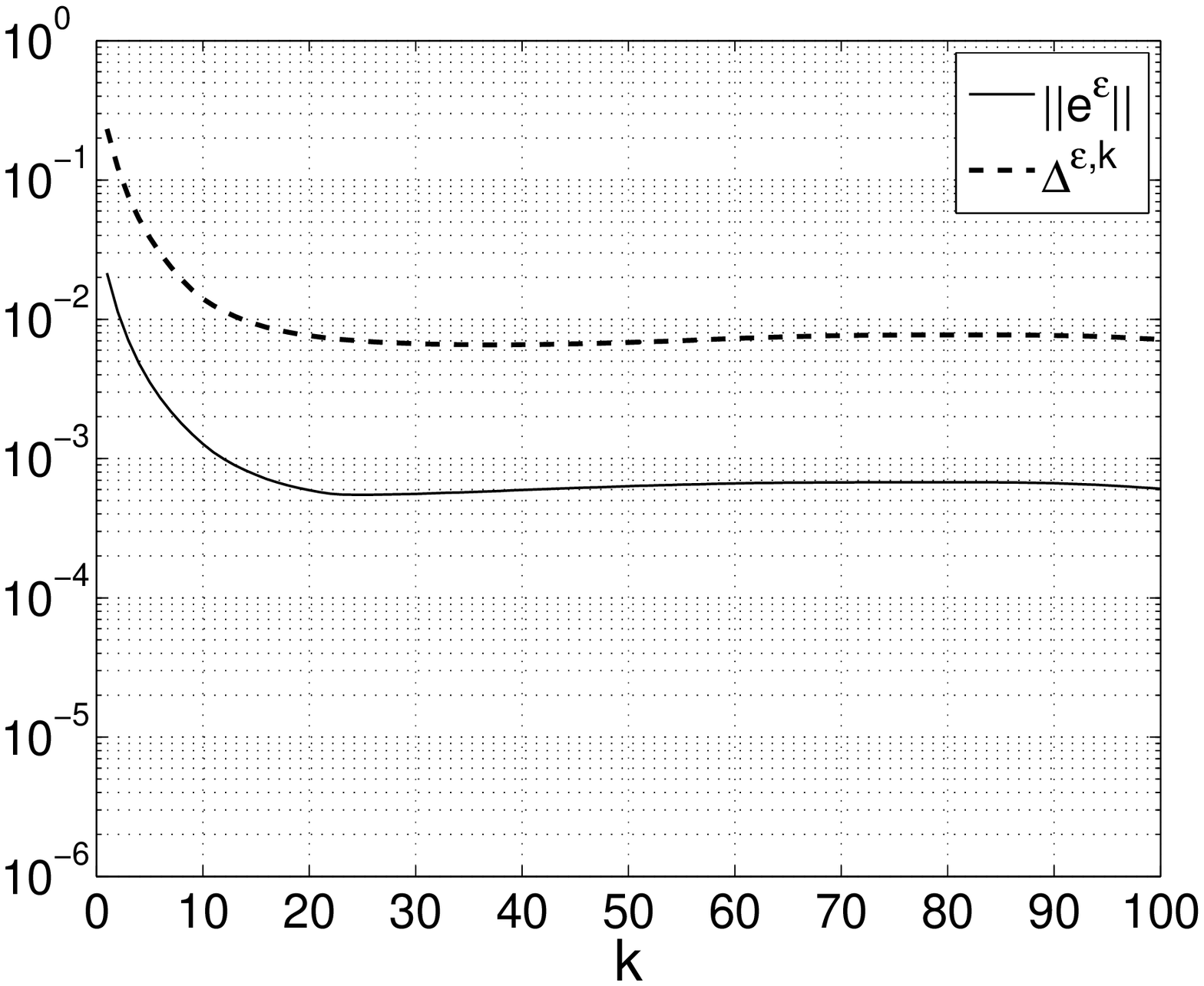,scale=0.25}
\end{minipage}\quad
\begin{minipage}[c]{0.35\textwidth}
\centerline{\footnotesize $N=28$ ($N_Z = 112$)}
\epsfig{file=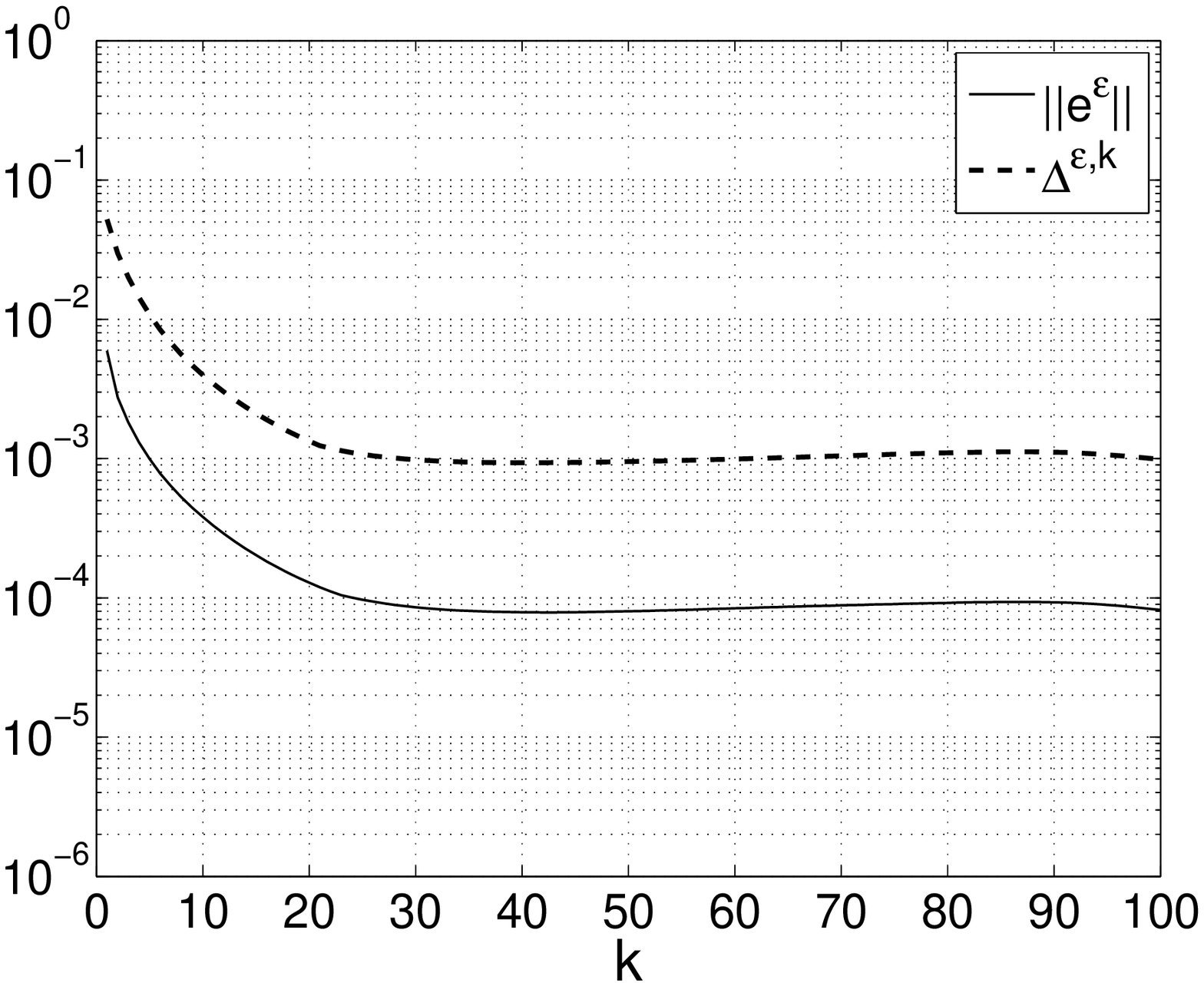,scale=0.25}
\end{minipage}}
\vspace{3ex}
\centerline{
\begin{minipage}[c]{0.15\textwidth}
\footnotesize (b)~$\varepsilon = 10^{-3}$
\end{minipage} \quad
\begin{minipage}[c]{0.35\textwidth}
\centerline{\footnotesize $N=13$ ($N_Z=70$)}
\epsfig{file=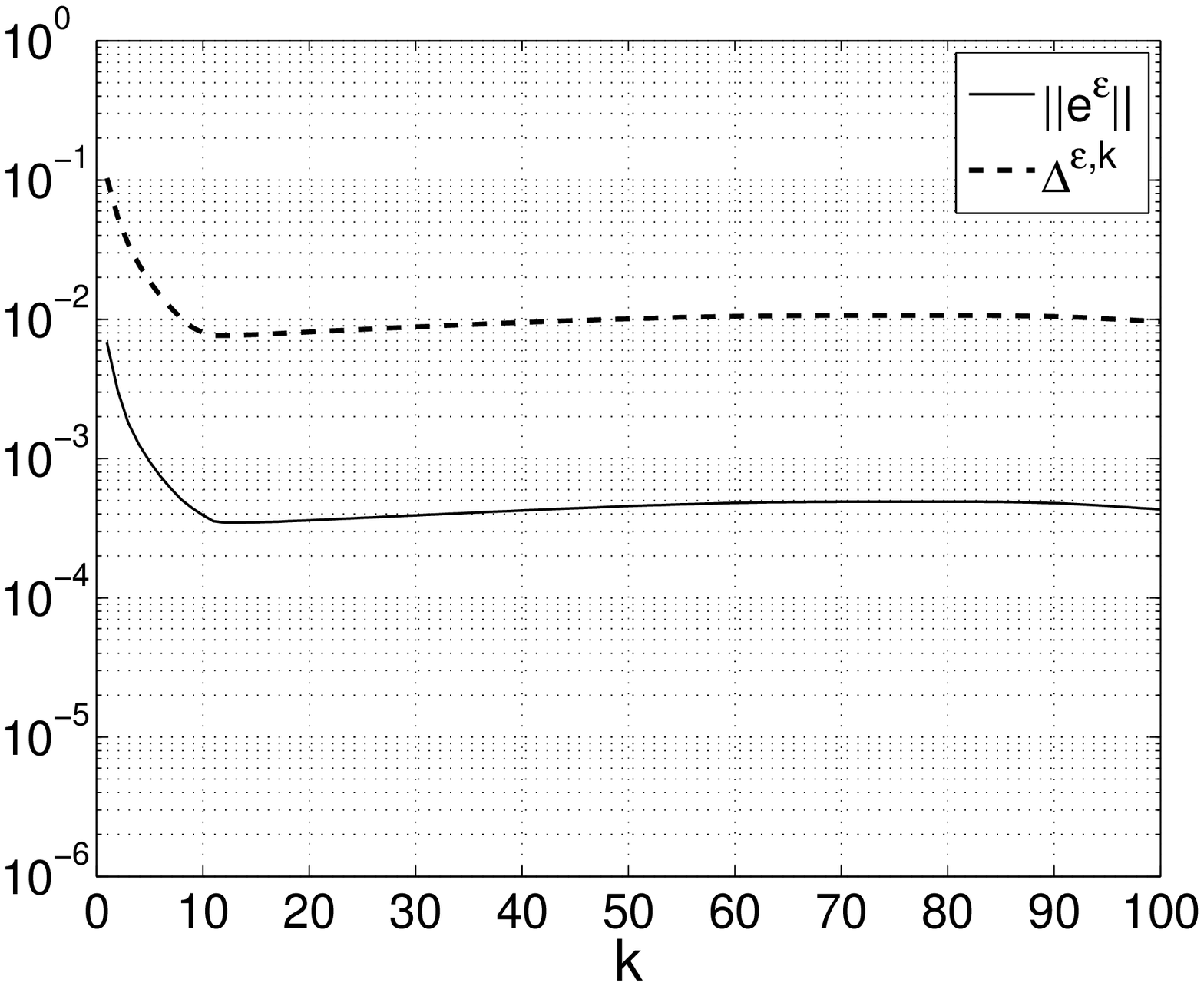,scale=0.25}
\end{minipage}\quad
\begin{minipage}[c]{0.35\textwidth}
\centerline{\footnotesize $N=20$ ($N_Z=107$)}
\epsfig{file=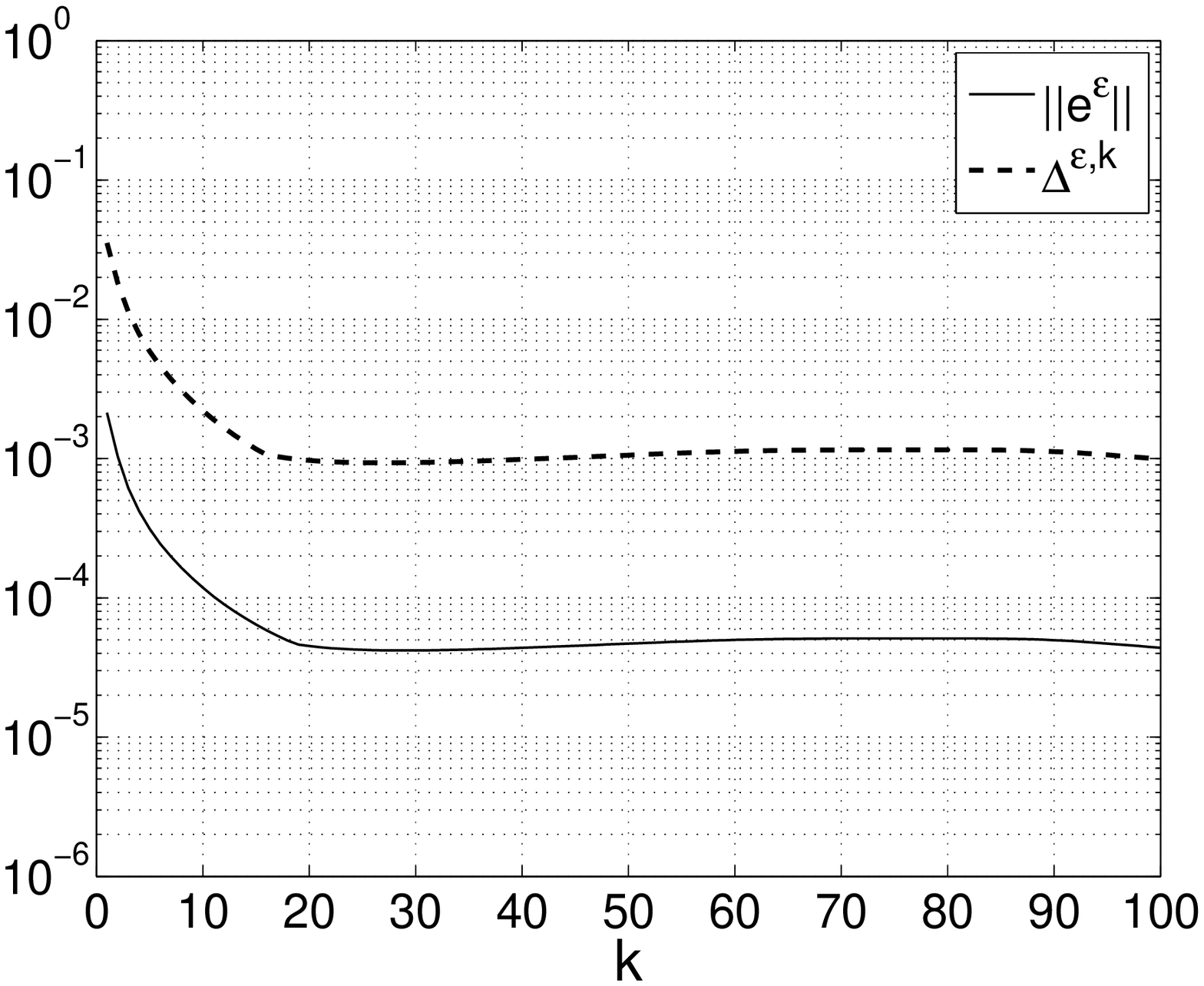,scale=0.25}
\end{minipage}}
\vspace{3ex}
\centerline{
\begin{minipage}[c]{0.15\textwidth}
\footnotesize (c)~$\varepsilon = 10^{-4}$
\end{minipage} \quad
\begin{minipage}[c]{0.35\textwidth}
\centerline{\footnotesize $N=14$ ($N_Z=79$)}
\epsfig{file=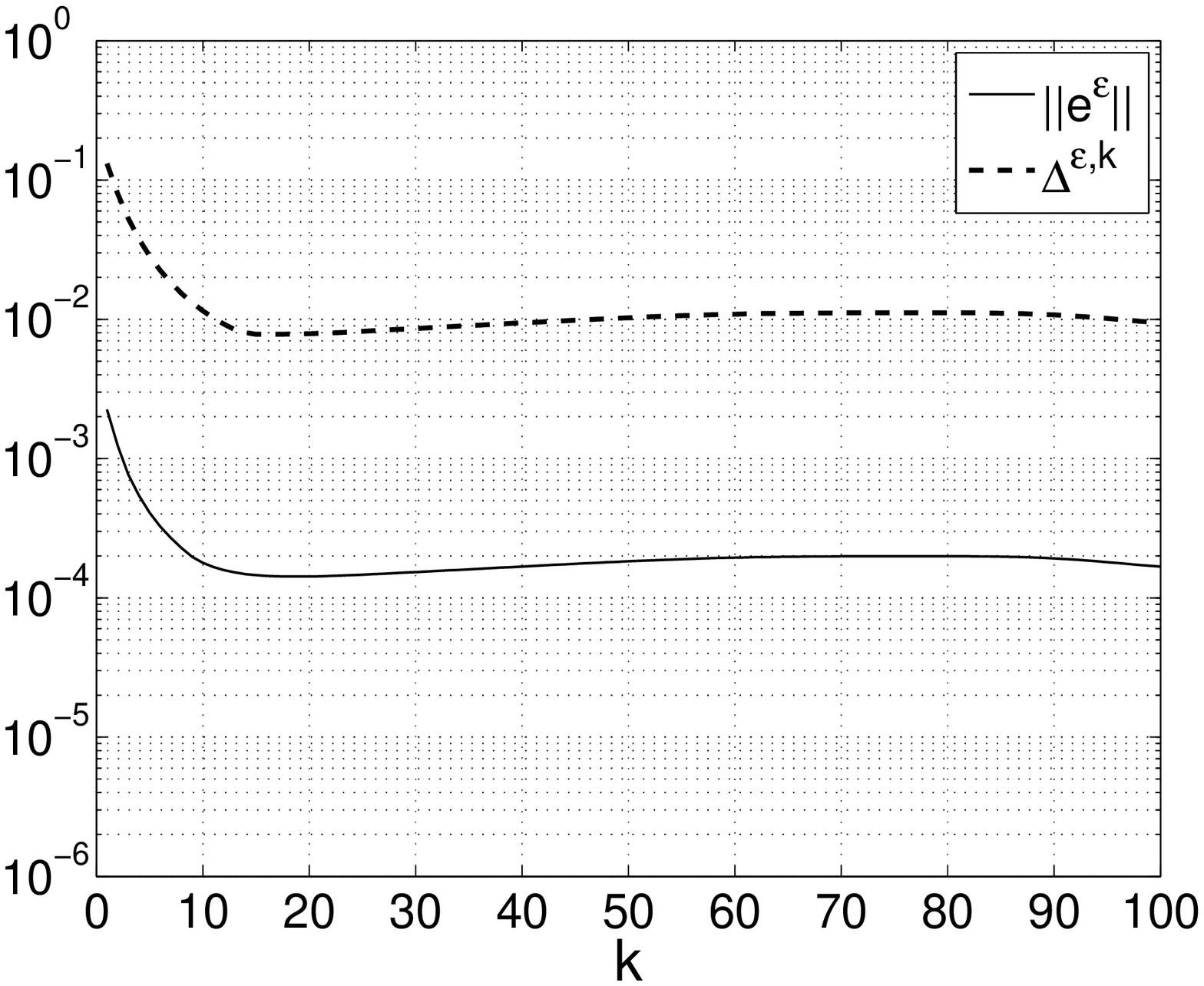,scale=0.25}
\end{minipage}\quad
\begin{minipage}[c]{0.35\textwidth}
\centerline{\footnotesize $N=25$ ($N_Z=150$)}
\epsfig{file=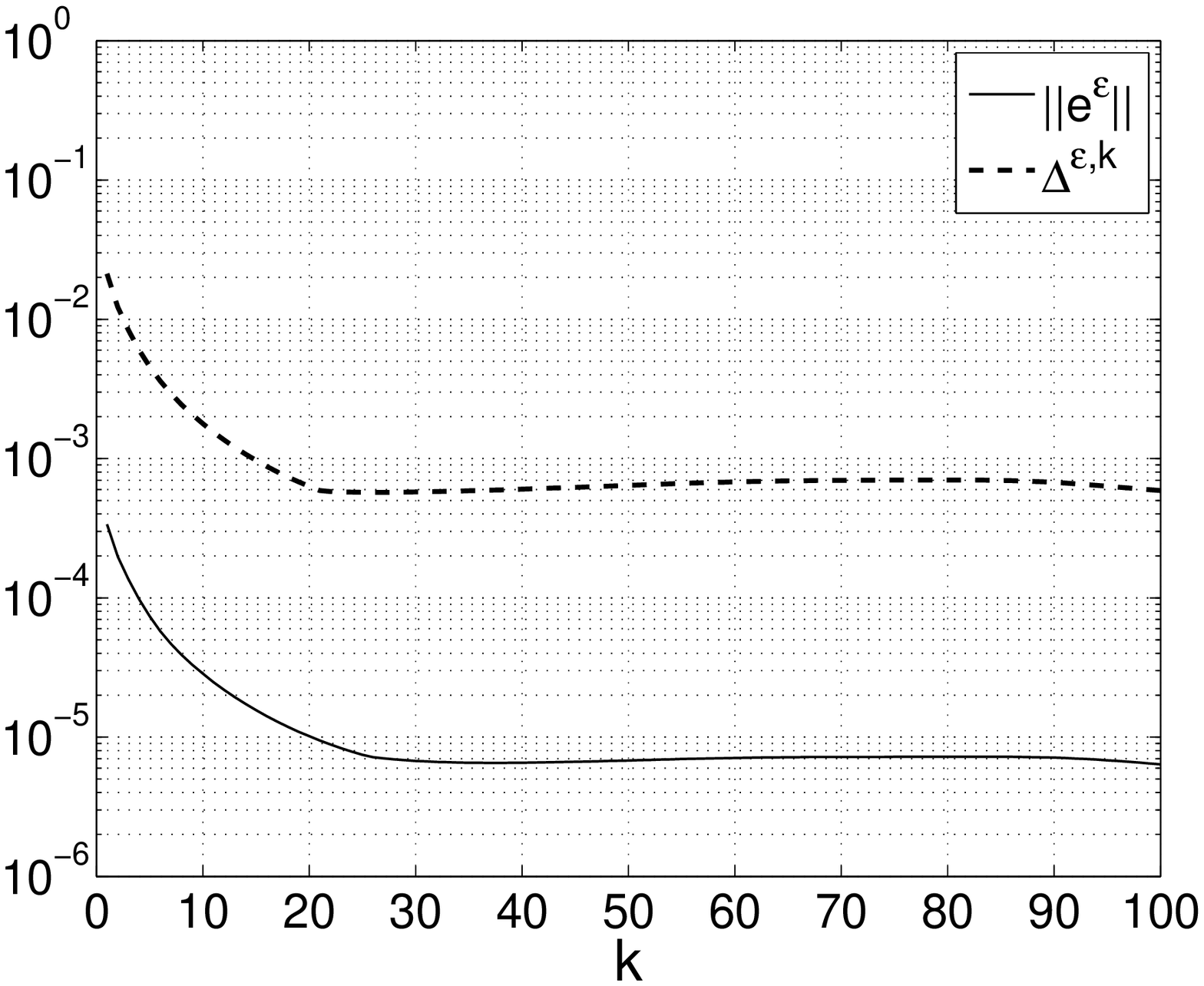,scale=0.25}
\end{minipage}}
\vspace{3ex}
\centerline{
\begin{minipage}[c]{0.15\textwidth}
\footnotesize (d)~$\varepsilon = 10^{-5}$
\end{minipage} \quad
\begin{minipage}[c]{0.35\textwidth}
\centerline{\footnotesize $N=21$ ($N_Z=121$)}
\epsfig{file=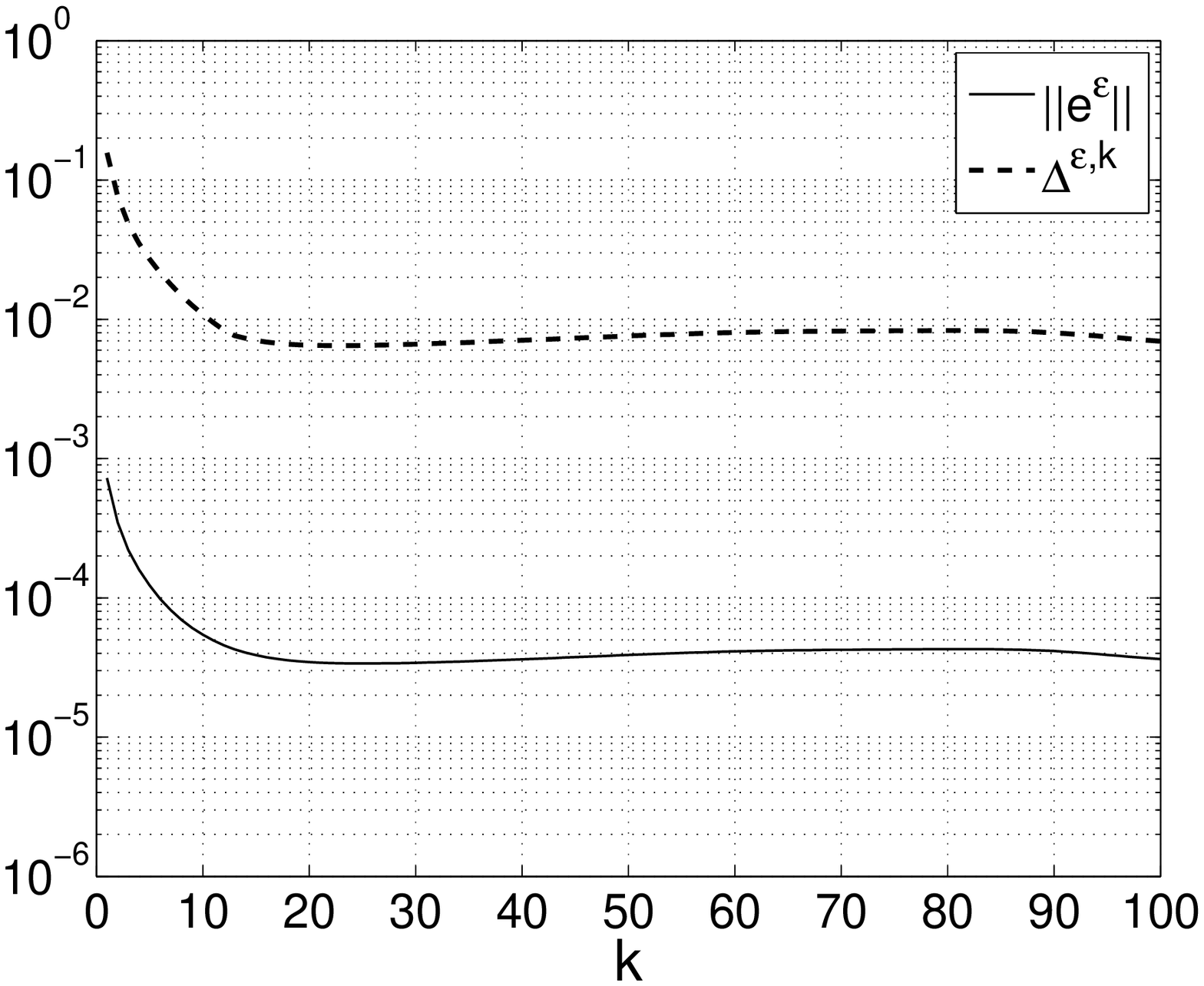,scale=0.25}
\end{minipage}\quad
\begin{minipage}[c]{0.35\textwidth}
\centerline{\footnotesize $N=31$ ($N_Z=174$)}
\epsfig{file=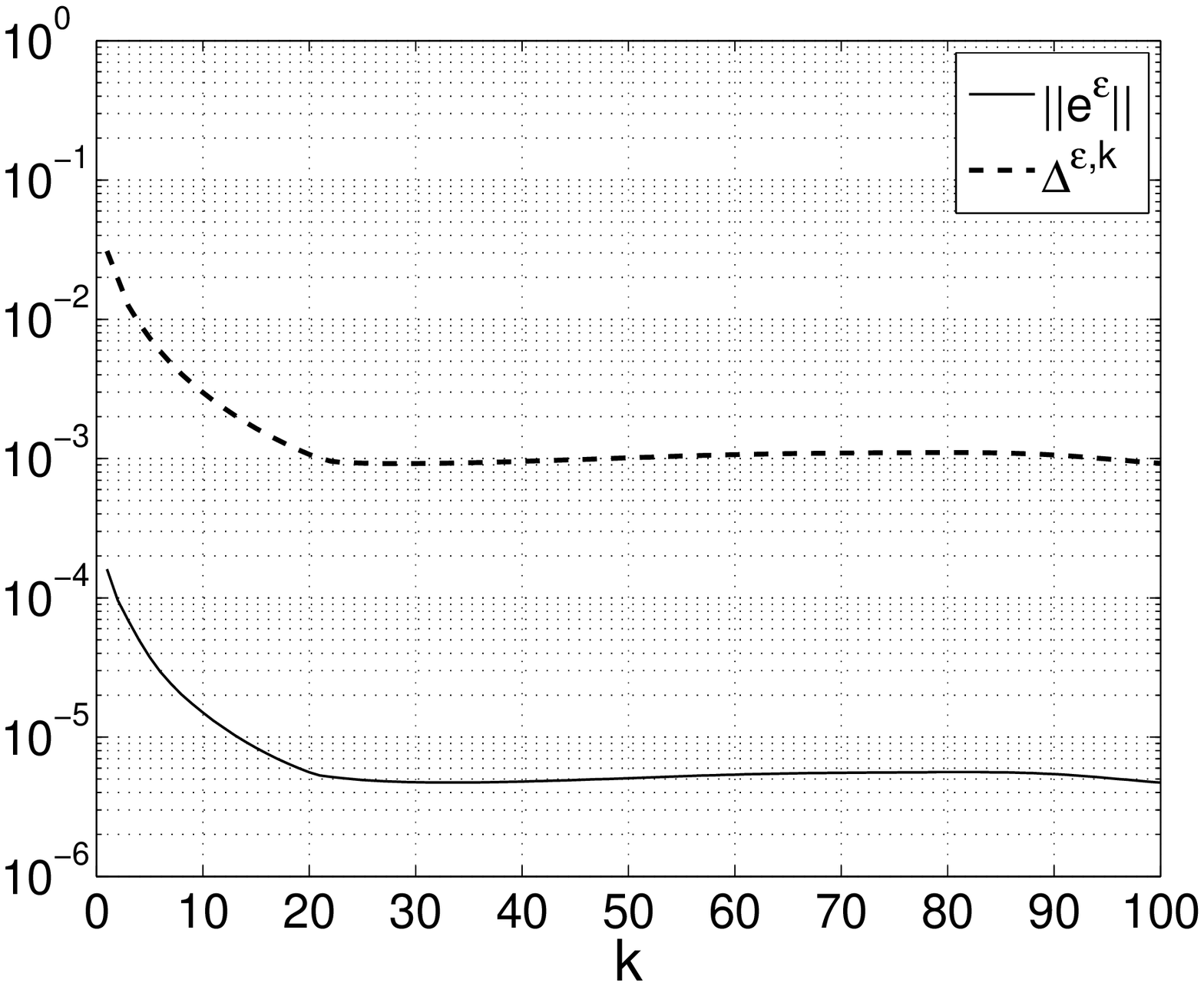,scale=0.25}
\end{minipage}}
\caption{Maximum error $\|e^{\varepsilon}_N(\mu)\|_{\ell^2(0,k;Z)}$ (see (\ref{eq:def_errors}), (\ref{eq:energy_norm_penalty})) and maximum error bound $\Delta^{\varepsilon,k}_N(\mu)$ (see (\ref{eq:energy_errbnd_penalty})) normalized with respect to $\|(u^{\varepsilon,j}(\mu),p^{\varepsilon,j}(\mu))_{j\in\bK}\|_{\ell^2(0,k;Z)}$ shown as functions of $k\in\bK$ for several values of $N$ for (a)~$\varepsilon=10^{-2}$, (b)~$\varepsilon=10^{-3}$, (c)~$\varepsilon=10^{-4}$, and (d)~$\varepsilon=10^{-5}$; the maximum is taken over 25 parameter values.}
\label{fig:err_errbnd_N_penalty}
\end{figure}

\begin{table}[htp]
\centering \footnotesize
\begin{tabular}{@{}c|c|c|c|c|c|c|c@{}}
\multicolumn{8}{l}{(a)~$\varepsilon = 10^{-2}$}\\[1ex]
\hline
$N$ & $N_Z$ & $k=10$ & $k=20$ & $k=40$ & $k=60$ & $k=80$ & $k=100$\\
\hline
& & & & & & &\\[-2.3ex]
$5$ & $20$ & $1.145\cdot 10^1$ & $1.282\cdot 10^1$ & $1.297\cdot 10^1$ & $1.296\cdot 10^1$ & $1.295\cdot 10^1$ & $1.293\cdot 10^1$\\
$15$ & $60$ & $1.202\cdot 10^1$ & $1.251\cdot 10^1$ & $1.292\cdot 10^1$ & $1.284\cdot 10^1$ & $1.281\cdot 10^1$ & $1.289\cdot 10^1$\\
$25$ & $100$ & $1.154\cdot 10^1$ & $1.154\cdot 10^1$ & $1.235\cdot 10^1$ & $1.248\cdot 10^1$ & $1.239\cdot 10^1$ & $1.227\cdot 10^1$\\
$35$ & $140$ & $1.132\cdot 10^1$ & $1.126\cdot 10^1$ & $1.163\cdot 10^1$ & $1.171\cdot 10^1$ & $1.159\cdot 10^1$ & $1.159\cdot 10^1$\\
$45$ & $180$ & $1.241\cdot 10^1$ & $1.235\cdot 10^1$ & $1.226\cdot 10^1$ & $1.206\cdot 10^1$ & $1.201\cdot 10^1$ & $1.194\cdot 10^1$\\
\hline
\end{tabular}\\[1ex]
\begin{tabular}{@{}c|c|c|c|c|c|c|c@{}}
\multicolumn{8}{l}{(b)~$\varepsilon = 10^{-3}$}\\[1ex]
\hline
$N$ & $N_Z$ & $k=10$ & $k=20$ &  $k=40$ & $k=60$ & $k=80$ & $k=100$\\
\hline
& & & & & & &\\[-2.3ex]
$4$ & $22$ & $2.691\cdot 10^1$ & $3.139\cdot 10^1$ & $3.273\cdot 10^1$ & $3.293\cdot 10^1$ & $3.275\cdot 10^1$ & $3.227\cdot 10^1$\\
$12$ & $66$ & $2.725\cdot 10^1$ & $2.969\cdot 10^1$ & $3.071\cdot 10^1$ & $3.166\cdot 10^1$ & $3.155\cdot 10^1$ & $3.101\cdot 10^1$\\
$19$ & $103$ & $2.358\cdot 10^1$ & $2.367\cdot 10^1$ & $2.593\cdot 10^1$ & $2.641\cdot 10^1$ & $2.645\cdot 10^1$ & $2.710\cdot 10^1$\\
$26$ & $145$ & $2.965\cdot 10^1$ & $2.933\cdot 10^1$ & $2.945\cdot 10^1$ & $2.973\cdot 10^1$ & $2.960\cdot 10^1$ & $2.983\cdot 10^1$\\
$34$ & $183$ & $2.983\cdot 10^1$ & $2.902\cdot 10^1$ & $2.903\cdot 10^1$ & $2.900\cdot 10^1$ & $2.850\cdot 10^1$ & $2.826\cdot 10^1$\\
\hline
\end{tabular}\\[1ex]
\begin{tabular}{@{}c|c|c|c|c|c|c|c@{}}
\multicolumn{8}{l}{(c)~$\varepsilon = 10^{-4}$}\\[1ex]
\hline
$N$ & $N_Z$ & $k=10$ & $k=20$ &  $k=40$ & $k=60$ & $k=80$ & $k=100$\\
\hline
& & & & & & &\\[-2.3ex]
$3$ & $18$ & $8.804\cdot 10^1$ & $8.882\cdot 10^1$ & $8.986\cdot 10^1$ & $9.150\cdot 10^1$ & $9.159\cdot 10^1$ & $9.113\cdot 10^1$\\
$10$ & $60$ & $6.726\cdot 10^1$ & $8.104\cdot 10^1$ & $9.668\cdot 10^1$ & $9.737\cdot 10^1$ & $9.639\cdot 10^1$ & $9.608\cdot 10^1$\\
$17$ & $101$ & $9.326\cdot 10^1$ & $1.010\cdot 10^2$ & $1.074\cdot 10^2$ & $1.055\cdot 10^2$ & $1.053\cdot 10^2$ & $1.072\cdot 10^2$\\
$24$ & $141$ & $1.058\cdot 10^2$ & $1.055\cdot 10^2$ & $1.043\cdot 10^2$ & $1.033\cdot 10^2$ & $9.825\cdot 10^1$ & $9.822\cdot 10^1$\\
$33$ & $181$ & $8.304\cdot 10^1$ & $8.290\cdot 10^1$ & $8.592\cdot 10^1$ & $8.635\cdot 10^1$ & $8.573\cdot 10^1$ & $8.706\cdot 10^1$\\
\hline
\end{tabular}\\[1ex]
\begin{tabular}{@{}c|c|c|c|c|c|c|c@{}}
\multicolumn{8}{l}{(d)~$\varepsilon = 10^{-5}$}\\[1ex]
\hline
$N$ & $N_Z$ & $k=10$ & $k=20$ &  $k=40$ & $k=60$ & $k=80$ & $k=100$\\
\hline
& & & & & & &\\[-2.3ex]
$3$ & $18$ & $2.706\cdot 10^2$ & $2.919\cdot 10^2$ & $3.311\cdot 10^2$ & $3.365\cdot 10^2$ & $3.199\cdot 10^2$ & $3.198\cdot 10^2$\\
$10$ & $55$ & $2.240\cdot 10^2$ & $2.510\cdot 10^2$ & $2.684\cdot 10^2$ & $2.699\cdot 10^2$ & $2.698\cdot 10^2$ & $2.701\cdot 10^2$\\
$17$ & $99$ & $2.696\cdot 10^2$ & $2.860\cdot 10^2$ & $3.103\cdot 10^2$ & $3.115\cdot 10^2$ & $3.107\cdot 10^2$ & $3.129\cdot 10^2$\\
$24$ & $138$ & $2.563\cdot 10^2$ & $2.950\cdot 10^2$ & $3.214\cdot 10^2$ & $3.206\cdot 10^2$ & $3.147\cdot 10^2$ & $3.210\cdot 10^2$\\
$32$ & $183$ & $2.786\cdot 10^2$ & $2.765\cdot 10^2$ & $3.206\cdot 10^2$ & $3.324\cdot 10^2$ & $3.306\cdot 10^2$ & $3.368\cdot 10^2$\\
\hline
\end{tabular}\\[1ex]
\caption{Maximum effectivities $\eta^{\varepsilon,k}_N(\mu)\equiv \Delta^{\varepsilon,k}_N(\mu)/\|e^{\varepsilon}_N(\mu)\|_{\ell^2(0,k;Z)}$ (see (\ref{eq:rig_energy_errbnd_penalty})) for several values of $k\in\bK$ and $N$ for (a)~$\varepsilon=10^{-2}$, (b)~$\varepsilon=10^{-3}$, (c)~$\varepsilon = 10^{-4}$, and (d)~$\varepsilon=10^{-5}$; the maximum is taken over 25 parameter values.}
\label{tbl:effectivities_penalty}
\end{table}

\begin{table}[htp]
\centering \footnotesize
\begin{tabular}{@{}c|c|c|c|c|c@{}}
\hline
& & & & & \\[-2.4ex]
 $\varepsilon$ & $N_Z$ &  $N$ & $(u^{\varepsilon,k}_N(\mu),p^{\varepsilon,k}_N(\mu)), k\in\bK$ & $\Delta^{\varepsilon,k}_N(\mu), k\in\bK$ & Total\\
 & & & & & \\[-2.3ex]
\hline
 & & & & & \\[-2.3ex]
$ 10^{-2}$ & 68 (112) & 17 (28) & 3.71 (7.65) &  14.53 (26.95) & 18.25 (34.60)\\
\hline
 & & & & & \\[-2.3ex]
$ 10^{-3}$ & 70 (107) & 13 (20) & 3.99 (7.43) & 17.19 (28.25) & 21.18 (35.67)\\
\hline
 & & & & & \\[-2.3ex]
$ 10^{-4}$ & 79 (150) & 14 (25) & 4.73 (13.59) & 20.28 (44.70) & 25.01 (58.29)\\
\hline
 & & & & & \\[-2.3ex]
$ 10^{-5}$ & 121 (174) & 21 (31) & 9.19 (17.47) & 33.81 (54.08) & 43.01 (71.55)\\
\hline
\end{tabular}\\[1ex]
\caption{Average computation times in milliseconds for the Online evaluation of $(u^{\varepsilon,k}_N(\mu),p^{\varepsilon,k}_N(\mu)), k\in\bK$, (assembly and solution of (\ref{eq:rb_scheme})) and the error bounds $\Delta^{\varepsilon,k}_N(\mu), k\in\bK$, (see (\ref{eq:energy_errbnd_penalty})) for different values of $\varepsilon$ with a prescribed accuracy of at least 1\% (resp., 0.1\%) for the RB approximations $(u^{\varepsilon,k}_N(\mu),p^{\varepsilon,k}_N(\mu)), k\in\bK$.}
\label{tbl:computation_times_penalty}
\end{table}

We close this section by discussing the Online computation times. For comparison, once the $\mu$-independent parts in the affine expansions of the involved operators have been formed (see \S \ref{ss:offline_online}), direct computation of the truth approximation $(u^{\varepsilon,k}(\mu),p^{\varepsilon,k}(\mu)), k\in\bK$, (i.e., assembly and solution of (\ref{eq:truth_scheme})) requires roughly 23 seconds on a 2.66 GHz Intel Core 2 Duo processor. Again, our rigorous and inexpensive RB {\em a posteriori} error bounds enable us to choose the RB dimension just large enough to obtain a desired accuracy.
Choosing $\varepsilon=10^{-2}$, the error bounds $\Delta^{\varepsilon,k}_N(\mu)$ are sharp with effectivities of approximately $12$ (see Table~\ref{tbl:effectivities_penalty}(a)) and prescribe a dimension of $N_Z= 68$ to achieve an accuracy of at least $1\%$ in the RB approximations $(u^{\varepsilon,k}_N(\mu),p^{\varepsilon,k}_N(\mu)), k\in\bK$ (see Fig.~\ref{fig:errbnd_alg6_penalty}). Once the database has been loaded, the Online calculation of $(u^{\varepsilon,k}_N(\mu),p^{\varepsilon,k}_N(\mu)), k\in\bK$, (i.e., assembly and solution of (\ref{eq:rb_scheme})) and $\Delta^{\varepsilon,k}_N(\mu), k\in\bK$, for any new value of $\mu\in\cD$ then takes on average 3.71 and 14.53 milliseconds, respectively, which is in total roughly 1,200 times faster than direct computation of the truth approximation. Choosing smaller values for $\varepsilon$, the error bounds become more pessimistic and thus dictate a larger system dimension at which they guarantee the same order of accuracy. For $\varepsilon=10^{-5}$, we need $N_Z=121$ to achieve a prescribed accuracy of at least $1\%$ in the RB approximations (see Fig.~\ref{fig:errbnd_alg6_penalty}); the Online calculation of $(u^{\varepsilon,k}_N(\mu),p^{\varepsilon,k}_N(\mu)), k\in\bK$, and $\Delta^{\varepsilon,k}_N(\mu), k\in\bK$, then takes on average 9.19 and 33.81 milliseconds, respectively, which is in total roughly 500 times faster than direct computation of the truth approximation. Thus, even for small penalty parameters $\varepsilon$, accurate approximations are guaranteed at significant Online savings.
Detailed computation times for different values of $\varepsilon$ are given in Table~\ref{tbl:computation_times_penalty}.

\section{Concluding Remarks} \label{s:conclusion}

In this paper, we present new RB methods for the instationary Stokes equations. 

Combining techniques developed in \cite{Gerner:2011fk,Gerner:2012fk} with current approaches for parabolic problems, we derive new rigorous {\em a posteriori} bounds for the errors in the RB velocity approximations and a POD greedy procedure that properly accounts for temporal and parametric causality as well as stability. The method provides rapidly convergent RB approximations that are highly efficient and whose accuracy is certified by sharp and inexpensive {\em a posteriori} error bounds.

An approximation by penalty or regularization allows for significant Offline savings at the expense of a less accurate truth approximation. Due to the introduced penalty term, an additional enrichment of the RB velocity approximation space is not always necessary to obtain stable approximations; moreover, we obtain {\em a posteriori} error bounds that do not involve the expensive computation of inf-sup stability constants. As in the stationary case (see \cite{veroy10:_stokes_penalty}), the method provides RB approximations and meaningful {\em a posteriori} error bounds that are computed very easily; nevertheless, drawbacks such as the disadvantageous dependence of the error bounds on the penalty parameter remain. 

Time integration is achieved through a backward Euler method.
Clearly, also other time integration schemes may be used. Using a Crank--Nicolson method, often preferred in practice due to its second-order accuracy, we may develop a penalty approach that is very similar to the one presented in this paper (see \cite{Gerner:2012ag}); in case of $\varepsilon=0$, useful RB {\em a posteriori} error bounds could not yet been derived and may therefore be---as well as {\em a posteriori} error bounds for the RB pressure approximations---part of future work.

\section*{Acknowledgments}
We would like to thank Prof.~Arnold Reusken and Prof. Martin Grepl of RWTH Aachen University for numerous very helpful suggestions and comments. We are also very grateful to Dr.~David J. Knezevic of Harvard University for his invaluable support on {\tt rbOOmit}~\cite{Knezevic:2011fk}, and Mark K\"{a}rcher of RWTH Aachen University for a very careful reading of the manuscript as well as many valuable comments and discussions.
Financial support from the Deutsche Forschungsgemeinschaft (German Research Foundation) through grant GSC 111 is gratefully acknowledged.
 
 \vfill
 

\bibliographystyle{siam}
\bibliography{bibtex/General,bibtex/Reduced_Basis,bibtex/Reduced_Order_Modelling,bibtex/Penalty}

\end{document}